\documentclass[11pt]{amsart}

\usepackage[latin1]{inputenc}
\usepackage{amssymb, amsmath, amsthm,mathrsfs}

\usepackage{amscd}
\usepackage{amssymb}
\usepackage{amsthm}
\usepackage{amsfonts}
\usepackage{amsmath}

\usepackage{verbatim} 

\usepackage{float}

\usepackage{hyperref}
\hypersetup{
    colorlinks,
    citecolor=black,%
    filecolor=black,%
    linkcolor=black,%
    urlcolor=black,
    pdfstartview={FitH},
    pdftitle={BSDE's driven by time-changed L{\'e}vy noises},
    pdfauthor={Steffen Sjursen and Giulia Di Nunno},
}

\theoremstyle{plain} \newtheorem{definisjon}{Definition}[section]
\theoremstyle{plain} \newtheorem{teorem}[definisjon]{Theorem}
\theoremstyle{plain} \newtheorem{lemma}[definisjon]{Lemma}
\theoremstyle{plain} 
\theoremstyle{definition} \newtheorem{remark}[definisjon]{Remark}
\theoremstyle{definition} 
\numberwithin{equation}{section}

\newcommand{\ie}{i.e.~}

\newcommand{\E}{\mathbb{E}}
\newcommand{\Prob}{\mathbb{P}}
\newcommand{\ind}{\mathbf{1}}
\newcommand{\Rr}{{\mathbb{R}_0}}
\newcommand{\RR}{{\mathbb{R}}}

\newcommand{\Tt}{[0,T]}
\newcommand{\Ff}{\mathcal{F}}
\newcommand{\Gg}{\mathcal{G}}

\newcommand{\FF}{\mathbb{F}}
\newcommand{\GG}{\mathbb{G}}
\newcommand{\Bb}{\mathcal{B}}
\newcommand{\Ht}{\tilde{H}}
\newcommand{\Ham}{\mathcal{H}}

\newcommand{\II}{\mathcal{I}}
\newcommand{\UU}{\mathcal{U}}
\newcommand{\LL}{\mathcal{L}}
\newcommand{\ins}{\,}
\newcommand{\minus}{\text{-}}
\DeclareMathOperator{\Hg}{\mathcal{H}^\Gg_2}
\DeclareMathOperator{\Af}{\mathcal{A}^\FF}
\DeclareMathOperator{\Ag}{\mathcal{A}^\GG}

\renewcommand{\theenumi}{\roman{enumi})}
\renewcommand{\labelenumi}{\theenumi}

\begin{document}
\title[BSDE's driven by time-changed L{\'e}vy noises]{BSDEs driven by time-changed L{\'e}vy noises and optimal control}
\date{\today }
\author[Di Nunno]{Giulia Di Nunno}
\address{Giulia Di Nunno: Center of Mathematics for applications, University of Oslo,
PO Box 1053 Blindern, N-0316 Oslo, Norway, and, Norwegian School of Economics and Business Administration, Helleveien 30, N-5045 Bergen, Norway.}
\author[Sjursen]{Steffen Sjursen}
\address{Steffen Sjursen: Centre of Mathematics for Applications , University of Oslo, PO Box 1053 Blindern, N-0316 Oslo, Norway}

\email[]{giulian\@@math.uio.no, s.a.sjursen\@@cma.uio.no}

\begin{abstract}
We study backward stochastic differential equations (BSDE's) for time-changed L{\'e}vy noises when the time-change is independent of the L{\'e}vy process. 
We prove existence and uniqueness of the solution and we obtain an explicit formula for linear BSDE's and a comparison principle. 
BSDE's naturally appear in control problems. Here we prove a sufficient maximum principle for a general optimal control problem of a system driven by a time-changed L{\'e}vy noise. As an illustration we solve the mean-variance portfolio
selection problem.
\end{abstract}


\maketitle

\section{Introduction}

We establish a framework for the study of backward stochastic differential equations (BSDE's) driven by a conditional Brownian motion and a doubly stochastic Poisson random field. Indeed the structure of these noises can be strongly related to the corresponding time-changed Browian motion and the time-changed Poisson random measure when the time-change is independent of the Brownian motion and Poisson field.

In the framework of the non-anticipating integration for martingale random fields, we prove the existence and uniqueness of the solution of a general BSDE of the form 
\begin{align}
Y_t &= \xi + \int\limits_t^T g_s\big(\lambda_s,Y_s,\phi_s \big) \ins ds - \int\limits_t^T\int\limits_{\mathbb{R}} \phi_s(z)  \ins \mu(ds,dz) \nonumber \\
&= \xi + \int\limits_t^T g_s\big(\lambda_s,Y_s,\phi_s \big) \ins ds - \int\limits_t^T \phi_s(0) \ins dB_s -\int\limits_t^T\int\limits_\Rr \phi_s(z) \ins \tilde{H}(ds,dz) 
\label{eq:BSDE_intro}
\end{align}
where $\mu$ is the mixture of a conditional Brownian measure $B$ on $[0,T]\times \{ 0 \}$ and a centered doubly stochastic Poisson measure $\Ht$ on $[0,T]\times \Rr$ ($\Rr := \RR \setminus \{0\})$. Namely 
\begin{equation}
\mu(\Delta) := B\big( \Delta \cap [0,T]\times \{ 0 \} \big) + \Ht \big( \Delta \cap [0,T]\times \Rr \big),
\end{equation}
for any Borel measurable set $\Delta$ in $[0,T] \times \RR$. Moreover we specifically study linear BSDE's achieving a closed form solution for the process $Y$ and use this solution to obtain a comparison theorem.

These results rely strongly on the stochastic integral representation of square integrable random variables and martingales. In the language of time-change, we can formulate the result as follows: Given the time-change processes $\Lambda^B$ and $\Lambda^H$, %of which we assume knowlegde, 
the complete filtered probability space $(\Omega, \Ff, \Prob, \GG)$ where $\GG$ is the filtration generated by $\mu$ and the whole of $\Lambda^B$ and $\Lambda^H$, any $L^2$-martingale $M$ can be represented as
\begin{equation}
M_t = M_0 + \int\limits_0^t\int\limits_{\mathbb{R}} \phi_s(z) \ins \mu(ds,dz) 
\label{eq:martingale_represention_introduction}
\end{equation}
where $\phi$ is proved to exist and $M_0$ is a random element measurable with respect to $\Lambda^B$ and $\Lambda^H$.

In \cite{Sjursen2011} a detailed study on the structure of the spaces generated by the measure $\Ht$ is carried though achieving chaos decompositions via orthogonal polynomials and also integral representation results of type \eqref{eq:martingale_represention_introduction} in which the integrand is given in closed form via the non-anticipating stochastic derivative in first place and then via Clark-Ocone type formulae and anticipating stochastic derivatives. These results  hold for very general choices of $\Lambda^H$ also beyond the present paper. Here we give an alternative slimmer proof for representation \eqref{eq:martingale_represention_introduction} which will provide only existence of the integrand $\phi$. This is enough for the study of \eqref{eq:BSDE_intro}.
% Here we tailor the study of \eqref{eq:martingale_represention_introduction} to the existence of the BSDE \eqref{eq:BSDE_intro}.

We remark that \eqref{eq:martingale_represention_introduction} shows that martingales $M$ of the type considered do not have a (full) predictable representation property as described in \cite{Bremaud1978,Nualart1995,Revuz1991} since the initial value $M_0$ is not a constant in general. Indeed the predictable representation property depends on the combination of integrator and the information flow. In \cite[Theorem 2.2]{DiNunno2007} it is proved that the predictable property with respect to the class of random measures $\mu$ with independent values \emph{if and only if} $\mu$ is given as a mixture of Gaussian and centered Poisson random measures. % Note that the sufficient statements for the Gaussian and the centered Poisson cases seperately are well known since long, see e.g. \textbf{Cite Ito and Dellacherie}.

The integration and the representation results are developed with respect to the filtration $\GG$, the filtration generated by $\mu$ and \emph{the entire history} of $\Lambda^B$ and $\Lambda^H$. It is with respect to %to the entire history of 
$\Lambda^B$ and $\Lambda^H$ that $H$ and $B$ have %the highly useful 
conditionally independent increments. From the point of view of modeling and applications $\GG$ is not a natural choice of filtration since it includes ``anticipating information'', the future values of $\Lambda^B$ and $\Lambda^H$. However we can still apply our results in the problems related to models where the reference filtration is $\FF$, the smallest right-continouos filtration to which $\mu$ is adapted. Indeed we show sufficient conditions for solving an optimal control problem with a classical performance functional for both $\GG$- and $\FF$-predictable controls. This is achieved by projecting the results obtained for the $\GG$-predictable case onto the $\FF$-predictable one.

\medskip

% In the case of point processes with the stochastic intensity measurable with respect to the filtration generated by the noise, it is also known that the predictable representation property holds, see e.g. \cite{Davis1976,Boel1975,Jacod1975}. However a full characterization of the random integrators for which the predictable representation property holds beyond the cases above is still missing.

The framework proposed based on specific integral representation under $\GG$ is a novel framework for problems related to time-changed processes. 
The work \cite{Lim2005} considers BSDE's with doubly stochastic Poisson processes, where the intensity of the doubly stochastic Poisson process depends on a Brownian motion in a specific way. Our setting does not overlap with that of \cite{Lim2005} due to a different relationship between the noises considered.
Our BSDE also differs from another approach to BSDE's beyond L{\'e}vy processes, \cite{Carbone2008, bobrovnytska2004, Oksendal2010, Jeanblanc2011,Kohlmann2010}, where an extra martingale $N$ is inserted to the backward stochastic differential equation so that $Y$ attains the terminal condition $Y_t = \xi$ and $Y_0$ is a real number. BSDE's with random measures is discussed in \cite{Jianming2000} assuming a martingale representation exist. We however prove the martingale representation and explicitly link the random measures, the martingale representation, and the conditions on the driver. %($g$ in \eqref{eq:BSDE_intro}).

% These time-changed processes are used in credit risk modeling (see e.g. \cite{Grandell1992, Lando1998}) and stochastic volatility models (see e.g. \cite{Barndorff-Nielsen2002,stein1991,Carr2003}). Thus the study of the related BSDE's gives the necessary tools in some relevant aspects of e.g. control and risk management. 

\medskip

Taking a view towards applications we sketch some of the uses of the time-changed L\'evy processes in mathematical finance and the relevance of our BSDE-framework. This is not meant as a comprehensive review. The time-changed L\'evy processes occur in mathematical finance in the modeling of asset prices as follows:
\begin{align}
dS_t &= S_{t\minus} \Big( \int\limits_\RR \psi_t(z) \ins \mu(dt,dz) \nonumber \\
&= S_{t\minus} \Big(  \psi_t(0) \ins dB_t + \int\limits_\Rr \psi_t(z) \ins \Ht(dt,dz) \Big) \quad S_0 >0.
\label{eq:asset_as_timechange}
\end{align}
%
% Examples of type \eqref{eq:asset_as_timechange} include stochastic volatility models like \cite{Barndorff-Nielsen2002,Carr2003,stein1991} and default risk as in \cite{Lando1998}. 

The well-known stochastic exponentiation model of \cite[Section 4.3]{Carr2003}, where stocks are modeled as time-changed pure jump L\'evy processes, can be described in our terminology as
\begin{equation}
S_t= \exp\Big\{ \int\limits_0^t \int\limits_\Rr z \ins \Ht(ds,dz) - \int\limits_0^t\int\limits_\Rr \big[ e^z-1-z \big] \ins \lambda^H_s \nu(dz) ds \Big\} 
\end{equation}
which in differential form is
\begin{equation}
dS_t = S_{t\minus} \Big( \int\limits_\Rr \big(e^z-1 \big) \ins \Ht(ds,dz) \Big).
\label{eq:Carr_model}
\end{equation}
Here the jump measure $\nu$ and time-change intensity $\lambda^H$ determine the properties of $S$.

A popular class of stochastic volatility models with Brownian filtrations including \cite{Barndorff-Nielsen2002,heston1993,Hull1987,stein1991} is
\begin{align}
dS_t &=  \rho S_{t\minus} \ins dt +   \sigma S_{t\minus} \lambda_t^B  \ins dW_t^{(1)} \label{eq:Brownian_stochastic_volatility}  \\
d\lambda^B_t &= M(\lambda_t^B) \ins dt + K(\lambda^B_t) \ins dW_t^{(2)} \label{eq:brownian_volatility}
\end{align}
where $M$ and $K$ are real functions, $\rho, \sigma \in \RR$ and $W^{(1)}$ and $W^{(2)}$ are Brownian motions. Here $S$ is the asset price and $\lambda^B$ the stochastic volatility.
Whenever $W^{(1)}$ and $W^{(2)}$ are \emph{independent}, $B_t:= \int_0^t \lambda_t^B \ins dW_t^{(1)}$ is a conditional independent Brownian motion as in Definition \ref{definisjon:listA} and our framework applies. 
%BSDE's for Brownian filtrations are well treated in the literature (see e.g. \cite{Karoui1997}). However BSDE equations where the noise from the volatility does not appear (in this case $W^{(2)}$) has recieved little attention. 

% but BSDE's in a filtration with time-change has recieved no attention outside \cite{Lim2005} 
%
%. Note however that we can deal with problems where the noise in the volatility is non-gaussian, like introducing jump terms in \eqref{eq:brownian_volatility}.
%
% \smallskip \noindent
% The stochastic volatility model of \cite{stein1991} is in our terminology given by
% \begin{align*}
% dS_t &= S_{t\minus} \sigma \ins dB_t \\
% d\lambda^B &= \alpha ( m - \lambda^B_t )\ins dt + K_2 \sqrt{\lambda^B}_t \ins dW_t 
% \end{align*}
% %
% where $m, K_2, \sigma$ are constants and $W$ a Brownian motion. Remark however that BSDE's for Brownian filtrations are well treated in the literature (see e.g \cite{Karoui1997}).

\smallskip \noindent
In credit risk, the jump times of the doubly stochastic Poisson process are used to signify the occurence of downwards abrupt price movements and default. A classical example \cite{Lando1998} is the case of an integer valued stochastic process $H_t$, $t\in [0,T]$, with $\nu(dz) = \ind_{\{z=1\}}(z)$ and $\lambda^H$ given. Then $\Ht_t = H_t -\Lambda^H_t$. 
% \ie that $H$ can be described as an integer valued stochastic process, $H_t$, $t\in[0,T]$. 
The default time $\tau$ is the first jump of $H$, ie $\tau = \inf_{t}\{ H(t) >0\}$. This is then used to model bonds or derivatives of the form $P \ind_{\tau>T}$, where $P$ is a random variable, so that $P \ind_{\tau>T}$ is a payoff which is received only if there is no default. An example of type \eqref{eq:asset_as_timechange} is the zero coupon bond which can be modeled as
\begin{equation*}
dS_t = S_{t\minus} \Big( \lambda^H_{t\minus} \ins dt -  \ins d \Ht_t \Big), \quad S_0=1, \; \text{ for } t\leq \tau. 
\end{equation*}

\medskip
% Our work is, to the best of our knowlegde, the only one using BSDE's to treat optimization problems of form \eqref{eq:performance_functional} with time-changed L\'evy processes. Furthermore it is, again t

To the best of our knowlegde, the present work is the first to detail BSDE's for time-changed L\'evy processes in general form, which opens up for studies on risk measures and filtration-consistent expectations as in \cite{Gianin2006,Royer2006} via our comparison theorem. 
Moreover we explicitly treat general optimal control problems with time-changed L\'evy processes, see e.g. \eqref{eq:performance_functional}, via the present BSDE. Indeed the BSDE can be used to investigate mean-variance hedging, utility maximization and optimal consumption problems for assets modeled as in \eqref{eq:asset_as_timechange} via Theorems \ref{teorem:optimal_control_GG} and \ref{teorem:optimal_control_FF}. 
% We refer to \cite{Framstad2004} for computational examples from the L\'evy case. 
%In particular, we can do this with time-changed L{\'e}vy processes \emph{with jumps} with very few assumptions on the model for the volatility. 
Utility maximization for time-changed L{\'e}vy processes is studied in \cite{kallsen2010a,kallsen2010b} for the power utility. Mean-variance hedging (for stochastic volatility and credit risk) has been discussed in terms of affine models \cite{Kallsen2009,Kallsen2011} and with BSDEs for general semi-martingales \cite{bobrovnytska2004,Jeanblanc2011,Kohlmann2010}. However \cite{bobrovnytska2004} only consider continuous semi-martingales, \cite{Jeanblanc2011} requires a system of several BSDEs while \cite{Kohlmann2010} requires a martingale representation result which is not true in our setting. 
%For the mean-variance problems we find alternate ways of describing the optimal controls, which in particular leads to an explicit formula for the optimal portfolio to the mean-variance selection problem. While for utility maximization we find a characterization of the optimal control in a more general setting.
% Our claim to novelty is that we find alternate ways of describing the optimal controls in more general settings, which in particular leads to a very explicit description to the mean-variance portfolio problem.
% Carbone2008t

\medskip

The present paper is organized as follows. In the next section the details about the noises considered and the integration framework are set into place. Section 3 is dedicated to the martingale representation type of result while section 4 deals with existence and uniqueness of the solution of the BSDE's \eqref{eq:BSDE_intro}. The study of explicit solutions of linear BSDE's and their applications to prove a comparison theorem is given in section 5. Finally we show a sufficient maximum principle in section 6 and we trace its use in some optimal control problems in section 7. There we study expected utility of the final wealth, for which we find a characterization of the optimal portfolio, and a mean-variance portfolio selection problem for which we give an explicit formula of the optimal portfolio.

\section{The framework}

\subsection{The random measures and their properties}

Let $(\Omega, \Ff, \Prob)$ be a complete probability space and $X := [0,T] \times \RR$, we will consider $X = \big ([0,T] \cup \{0\} \big) \cup \big( [0,T]\times \Rr\big)$, where $\Rr = \mathbb{R}\setminus \{0\}$ and $T>0$. Denote $\Bb_X$ the Borel $\sigma$-algebra on $X$. %and $\Bb_X^c$ is the elements of $\Bb_X$ with compact closure. We recall that $B_X$ is separable and is generated by $\Bb_X^c$.
Throughout this presentation $\Delta \subset X$ denotes an element $\Delta$ in $\Bb_X$. 

Let $\lambda:= (\lambda^B, \lambda^H)$ be a two dimensional stochastic process such that % both $\lambda^B$ and $\lambda^H$ satisfy 
each component $\lambda^l$, $l=B,H$, satisfies
\begin{enumerate}
\item $\lambda_t^l \geq 0$ $\Prob$-a.s. for all $t\in [0,T]$,
% $\mathbb{P} \big( \lambda_t < 0 ) = 0 $ for all $t \in \Tt$.
\label{item:lambda_1}
\item $\lim_{h\to 0} \mathbb{P} \big( \big| \lambda_{t+h}^l-\lambda_t^l \big| \geq \epsilon \big) = 0$ 
for all $\epsilon > 0$ and almost all $t \in \Tt$,
\label{item:lambda_2}
\item %\begin{equation} 
$\E \big[ \int_0^T  \lambda_t^l  \ins dt \big] < \infty, $ % ,\quad \text{for all } C \in \mathcal{B}_\Tt^c.
%\label{eq:lambda_finite_E}
%\end{equation} 
\label{item:lambda_3}
\end{enumerate}
We denote $\LL$ as the space of all processes $\lambda:= (\lambda^B, \lambda^H)$ satisfying \ref{item:lambda_1}-\ref{item:lambda_2}-\ref{item:lambda_3} above.

% (\ie \ref{item:lambda_1}-\ref{item:lambda_2}-\ref{item:lambda_3} also holds for $\lambda^H$). \textbf{Remark that $\lambda$ isn't continouos in probability since the continuity only holds for almost all t}

Define the random measure $\Lambda$ on $X$ by 
% \begin{equation}
% \Lambda(\Delta) = \int\int\limits_{\Delta}  \delta_{0}(dz) \ins  \lambda_t^B \ins dt + \int\int\limits_{\Delta} \ind_{\Rr}(z) \nu(dz) \ins \lambda^H_t \ins dt 
% \Lambda(\Delta) := \int\limits_{\Delta} \ind_{\{0\}}(z) \ins \lambda_t^B dt + \int\limits_\Delta \ind_{\Rr}(z) \ins \nu(dz) \lambda^H_t dt,
% \end{equation}
\begin{equation}
\Lambda(\Delta) := \int\limits_0^T \ind_{\{(t,0) \in \Delta  \}}(t) \ins \lambda_t^B dt + \int\limits_0^T \int\limits_{\Rr} \ind_\Delta (t,z) \ins \nu(dz) \lambda^H_t dt, 
\label{eq:measure_Lambda}
\end{equation}
as the mixture of measures on disjoint sets. Here $\nu$ is a deterministic, $\sigma$-finite measure on the Borel sets of $\Rr$ satisfying 
\begin{equation*}
\int_\Rr z^2 \ins \nu(dz)<\infty. 
\end{equation*}
We denote the $\sigma$-algebra generated by the values of $\Lambda$  by $\Ff^\Lambda$.
% We let $\Ff^\Lambda$ be the $\sigma$-algebra generated by $\Lambda$. 
Furthermore, $\Lambda^H$ denotes the restriction of $\Lambda$ to $[0,T]\times \Rr$ and $\Lambda^B$ the restriction of $\Lambda$ to $[0,T]\times \{0\}$. Hence $\Lambda(\Delta)=\Lambda^B(\Delta\cap[0,T]\times\{0\})+\Lambda^H(\Delta\cap[0,T]\times\Rr)$, $\Delta\subseteq X$. 
Here below we introduce the noises driving \eqref{eq:BSDE_intro}.

% Denote $\Phi$ as the cumulative probability distribution of a standard normal random variable.

\begin{definisjon}
$B$ is a signed random measure on the Borel sets of $\Tt \times \{0\}$ satisfying, %with $\Phi$ being the cumulative probability distribution function of a standard normal random variable,
\begin{enumerate}
\renewcommand{\theenumi}{A\arabic{enumi})}
\item $\Prob\Big( B(\Delta) \leq x \,\Big| \Ff^\Lambda \Big) = \Prob\Big( B(\Delta) \leq x  \,\Big| \Lambda^B(\Delta) \Big) =  \Phi \big(\frac{x}{\sqrt{ \Lambda^B(\Delta)}}\big)$, $x\in\mathbb{R}$, $\Delta \subseteq [0,T]\times\{0\}$,
\label{list:A1}
\item $B(\Delta_1)$ and $B(\Delta_2)$ are conditionally independent given $\Ff^{\Lambda}$ whenever $\Delta_1$ and $\Delta_2$ are disjoint sets. 
\label{list:A2}
\end{enumerate}
Here $\Phi$ stands for the cumulative probability distribution function of a standard normal random variable.

$H$ is a random measure on the Borel sets of $\Tt\times \Rr$ satisfying
\begin{enumerate}
\renewcommand{\theenumi}{A\arabic{enumi})}
\setcounter{enumi}{2}
% \item $\Prob\Big( H(B) = k\Big) = \E \Big[ \frac{\alpha(B)^k}{k!} e^{-\alpha(B)} \Big]$, 
% \label{list:A1}
\item $\Prob\Big( H(\Delta) = k \,\Big|\Ff^\Lambda \Big) = \Prob\Big( H(\Delta) = k \,\Big| \Lambda^H(\Delta) \Big)  =  \frac{\Lambda^H(\Delta)^k}{k!} e^{-\Lambda^H(\Delta)}$, $k\in \mathbb{N}$, $\Delta \subseteq \Tt\times \Rr$,
% \item $\Prob\Big( H(\Delta) = k \,\Big| \alpha(\Delta) \Big) = \Prob\Big( H(\Delta) = k \,\Big| \Ff^{\Lambda^H} \Big) =  \frac{\alpha(\Delta)^k}{k!} e^{-\alpha(\Delta)}$, 
\label{list:A3}
\item $H(\Delta_1)$ and $H(\Delta_2)$ are conditionally independent given $\Ff^{\Lambda}$ whenever $\Delta_1$ and $\Delta_2$ are disjoint sets. 
\label{list:A4}
\end{enumerate}

Furthermore we assume that 
\begin{enumerate}
\renewcommand{\theenumi}{A\arabic{enumi})}
\setcounter{enumi}{4}
\item $B$ and $H$ are conditionally independent given $\Ff^\Lambda$.
\label{list:A5}
\end{enumerate}
\label{definisjon:listA}
\end{definisjon}

% Naturally when considering $B(\Delta)$ and $H(\Delta)$ it is assumed that $\Delta$ is a subset of either $[0,T]\times\{0\}$ or $[0,T]\times\Rr$.

Conditions \ref{list:A1} and \ref{list:A3} mean that conditional on $\Lambda$, $B$ is a Gaussian random measure and  $H$ is Poisson a random measure. % To be specific $\E[ B(\Delta) | \Lambda^B(\Delta) ]$ is normally distributed with mean $0$ and variance $\Lambda^H(\Delta)$ while $\E[ H(\Delta) | \Lambda^H(\Delta) ]$ is Poisson distributed with parameter $\Lambda^H(\Delta)$. 
In particular, if $\lambda^B$ and $\lambda^H$ are deterministic then $B$ is a Wiener process and $H$ is a Poisson random random measure.

We refer to \cite{Grigelions1975} or \cite{Kallenberg1997} for the existence of conditional distributions as in Definition \ref{definisjon:listA}.

Let $\Ht:= H-\Lambda^H$ be the signed random measure given by
\begin{equation*}
\Ht(\Delta) = H(\Delta) - \Lambda^H(\Delta), \quad \Delta \subset \Tt\times \Rr.
\end{equation*}
\begin{definisjon}
We define the signed random measure $\mu$ on the Borel subsets of $X$ by
\begin{equation}
\mu(\Delta) := B\Big(\Delta \cap[0,T]\times \{0\} \Big) + \Ht\Big(\Delta \cap [0,T]\times \Rr\Big), \quad \Delta \subseteq X.
\label{eq:mu_definition}
\end{equation}
\end{definisjon}
Clearly, from \ref{list:A1} we have that the conditional first moment of $B$ is $\E \big[ B(\Delta) \big| \Ff^\Lambda \big] = 0$ and from \ref{list:A3} the conditional first moment of $H$ is $\E \big[ H (\Delta) \big| \Ff^\Lambda \big] = \Lambda^H(\Delta)$ so that $\E \big[ \Ht (\Delta) \big| \Ff^\Lambda \big] = 0$.
Thus
\begin{equation}
\E \big[ \mu (\Delta) \big| \Ff^\Lambda \big] = 0.
\label{eq_first_moment_mu}
\end{equation}
The second conditional moments of $B$ and $\Ht$ are given by 
\begin{align*}
\E \big[ B (\Delta)^2 \big| \Ff^\Lambda \big] &= \Lambda^B(\Delta), \\
\E \big[ \Ht (\Delta)^2 \big| \Ff^\Lambda \big] &= \Lambda^H(\Delta).
\end{align*}
By the conditional independence \ref{list:A2}, \ref{list:A4} and \ref{list:A5} we have
\begin{equation*}
\E \Big[ \mu(\Delta)^2 \Big|  \Ff^\Lambda \Big] = \Lambda(\Delta)
\end{equation*}
and
\begin{equation}
\E\big[ \mu(\Delta_1) \mu(\Delta_2)  \big| \Ff^\Lambda \big] = \E\big[ \mu(\Delta_1)  \big| \Ff^\Lambda \big] \E\big[ \mu(\Delta_2)  \big| \Ff^\Lambda \big] = 0
\label{eq:conditional_independence_easy}
\end{equation}
for $\Delta_1$ and $\Delta_2$ disjoint. Hence $\mu(\Delta_1)$ and $\mu(\Delta_2)$ are conditionally orthogonal.

The random measures $B$ and $H$ are related to a specific form of time-change for Brownian motion and pure jump L{\'e}vy process. More specifically define $B_t := B( [0,t]\times \{ 0\})$, $\Lambda^B_t := \int_0^t \lambda^B_s \ins ds$, $\eta_t := \int_0^t\int_{\Rr} z \ins \Ht(ds,dz)$ and $\hat{\Lambda}^H_t := \int_0^t \lambda_s^H \ins ds$, for $t \in [0,T]$.

We can immediately see the role that the time-change processes $\Lambda^B$ and $\hat{\Lambda}^H$ play, studying the characteristic function of $B$ and $\eta$. In fact, from \ref{list:A1} and \ref{list:A3} we see that the conditional characteristic functions of $B_t$ and $\eta_t$ are given by
\begin{align}
\E \big[ e^{ic B_t} \big| \Ff^\Lambda \big] &= \exp\big\{ \int\limits_0^t \frac{1}{2}c^2  \ins \lambda^B_s ds \big\} = \exp\big\{ \frac{1}{2}c^2 \Lambda_t^B \big\}, \quad c\in \RR
\label{eq:characteristic_B} \\
\E \big[ e^{ic \eta_t} \big| \Ff^\Lambda \big] &= \exp\big\{\int\limits_0^t \int\limits_\Rr \big[e^{icz} -1-icz \big]\ins \nu(dz) \lambda_t^H dt \big\} \nonumber \\
&= \exp\big\{ \big( \int\limits_\Rr \big[e^{icz} -1-icz \big]\ins \nu(dz) \big) \, \hat{\Lambda}_t^H \big\}, \quad c \in \RR.
\label{eq:characteristic_H} 
\end{align}
%
% \textbf{The characteristic function is also useful to easily connect to other works like some of the time-changes described by Carr et al}

Indeed there is a strong connection between the distributions of $B$ and the Brownian motion and between $\eta$ and a centered pure jump L{\'e}vy process with the same jump behavior. The relationship is based on a random distortion of the time scale. The following characterization is due to \cite[Theorem 3.1]{Serfozo1972} (see also \cite{Grigelions1975}). 
\begin{teorem}
Let $W_t$, $t\in \Tt$ be a Brownian motion and $N_t$, $t\in\Tt$ be a centered pure jump L\'evy process with Levy measure $\nu$. Assume that both $W$ and $N$ are \emph{independent} of $\Lambda$. 
Then $B$ satisfies \ref{list:A1}-\eqref{eq:characteristic_B} and \ref{list:A2} if and only if, for any $t\geq 0$,
\begin{equation*}
B_t  \stackrel{d}{=} W_{\Lambda_t^B},
\end{equation*}
and $\eta$ satisfies \ref{list:A3}-\eqref{eq:characteristic_H} and \ref{list:A4} if and only if, for any $t\geq 0$,
\begin{equation*}
\eta_t \stackrel{d}{=} N_{\hat{\Lambda}^H_t}. 
\end{equation*}
\end{teorem}
In addition, $B$ is infinitely divisible if $\Lambda^B$ is infinitely divisible and $\eta$ is infinitely divisible if $\hat{\Lambda}^H$ is infinitely divisible, see \cite[Theorem 7.1]{Barndorff-Nielsen2006}.

\subsection{Stochastic non-anticipating integration}
\label{section:integration}

Let us define $\FF^{\mu}= \{\Ff_t^{\mu}$,\; $t\in [0,T]\}$ as the filtration generated by $\mu(\Delta)$, $\Delta\subset [0,t]\times \RR$. In view of \eqref{eq:mu_definition}, \ref{list:A1}, and \ref{list:A3} we can see, that for any $t\in [0,T]$, 
\begin{equation*}
\Ff_t^\mu = \Ff_t^B \vee \Ff_t^H \vee \Ff_t^\Lambda, 
\end{equation*}
where $\Ff_t^B$ is generated by $B(\Delta\cap [0,T]\times \{0\})$, $\Ff_t^H$ by $H(\Delta\cap [0,T]\times \Rr)$, and $\Ff^\Lambda_t$ by $\Lambda(\Delta)$, $\Delta \in [0,t]\times \RR$. This is an application of \cite[Theorem 1]{Winkel2001} and \cite[Theorem 2.8]{Sjursen2011}. Set $\FF= \{\Ff_t$,\; $t\in [0,T]\}$ where
\begin{equation*}
\Ff_t = \bigcap_{r>t} \Ff^{\mu}_r 
\end{equation*}
Furthermore, we set $\GG=\{\Gg_t,\; t\in [0,T]\}$ where $\Gg_t=\Ff_t^{\mu}\vee \Ff^{\Lambda}$.
Remark that $\Gg_T = \Ff_T$, $\Gg_0=\Ff^\Lambda$, while $\Ff_0^{\mu}$ is trivial. From now on we set $\Ff=\Ff_T$.

\begin{lemma}
The filtration $\GG$ is right-continouos. 
\end{lemma}
\begin{proof}
This can be shown using classical arguments from the L\'evy case (as in e.g. \cite[Theorem 2.1.9]{Applebaum2004}).
\end{proof}

For $\Delta \subset (t,T]\times \RR$, the conditional independence \ref{list:A2}, \ref{list:A4} means that
\begin{equation}
\E\big[ \mu(\Delta) \ins \big| \Gg_t \big] = \E\big[ \mu(\Delta) \ins \big| \Ff_t \vee \Ff^\Lambda \big] = \E\big[ \mu(\Delta) \ins \big| \Ff^\Lambda \big] =0.
\label{eq:conditional_independence_explained}
\end{equation}
Thus $\mu$ has the martingale property with respect to $\GG$ from \eqref{eq_first_moment_mu}.
% \begin{equation}
% \E\Big[ \mu(\Delta) \Big| \Ff_t \Big] = \E\Big[ E [ \mu(\Delta) | \Gg_t ]  \Big| \Ff_t \Big] = 0, \quad \Delta \subset (t,T]\times \RR
% \label{eq:mu_martingale_property}
% \end{equation}
%
Hence $\mu$ is a martingale random field with respect to $\GG$ in the sense of \cite{DiNunno2010} since 
\begin{itemize}
\item $\mu$ has a $\sigma$-finite variance measure 
\begin{equation*}
m(\Delta) := \E \big[ \mu(\Delta)^2] =  \E \big[ \Lambda(\Delta)],
\end{equation*}
\item $\mu$ is $\GG$-adapted, % to both $\GG$ and $\FF$,
\item $\mu$ has conditionally orthogonal values, if $\Delta_1, \Delta_2 \subset (t,T]\times \RR$ such that $\Delta_1 \cap \Delta_2 = \emptyset$ then, combining the arguments in \eqref{eq:conditional_independence_easy} and \eqref{eq:conditional_independence_explained},
\begin{equation}
\E \Big[ \mu(\Delta_1) \mu(\Delta_2) \ins \Big| \Gg_t \Big] = \E \Big[ \mu(\Delta_1) \ins\Big| \Ff^\Lambda \Big] \E \Big[ \mu(\Delta_2)\ins \Big| \Ff^\Lambda \Big] =0.
\label{eq_mu:conditionally_orthogonal_values} 
\end{equation}
\end{itemize}
Denote $\II$ as the subspace of $L^2(\Tt\times \RR \times \Omega, \Bb_X \times \Prob, \Lambda\times \Prob)$ of the random fields admitting a $\GG$-predictable modification, in particular % $\phi: \Omega \times [0,T]\times Z \to \mathbb{R}$ such that
\begin{equation}
\| \phi \|_{\II} :=  \Big( \E \Big[ \int\limits_0^T \phi_s(0)^2 \ins \lambda_s^B ds + \int\limits_0^T\int\limits_\Rr \phi_s(z)^2 \ins \nu(dz) \lambda^H_s ds \Big] \Big) ^{\frac{1}{2}} < \infty.
\label{eq:II_defined}
\end{equation}
For any $\phi \in \II$ we define the (It\^o type) non-anticipative stochastic integral $I: \II \Rightarrow L^2(\Omega, \Ff, \Prob)$ by
\begin{equation*}
I(\phi ) := \int\limits_0^T \phi_s(0) \ins dB_s + \int\limits_0^T\int\limits_\Rr \phi_s(z) \ins \Ht(ds,dz).
\end{equation*}
We refer to \cite{DiNunno2010} for details on the integration with respect to martingale random fields of the type discussed here. In particular, $I$ is a linear isometric operator:
\begin{equation}
\Big( E\big[ I(\phi)^2 \big]\Big) ^{\frac{1}{2}} = \|I(\phi) \|_{L^2(\Omega, \Ff, \Prob)} = \| \phi \|_{\II}.
\label{eq:isometry_general}
\end{equation}
Because of the structure of the filtration considered we have:
% An unusual property of these integrals is that they allow the interchange of integration and random variables in the following sense:

\begin{lemma}
Consider $\xi \in L^2\big(\Omega,\mathcal{F}^\Lambda,\mathbb{P}\big)$ and $\phi \in \mathcal{I}$. Then
\begin{equation*}
\xi I(\phi) = I( \xi\phi),
\end{equation*}
whenever either side of the equality exists as an element in $L^2\big(\Omega,\mathcal{F},\mathbb{P}\big)$.
\label{lemma:f_alpha_and integration}
\end{lemma}

\begin{proof}
Assume that $\xi$ is bounded and $\phi \in \II$ is simple, \ie 
\begin{equation*}
\phi_s(z,\omega) = \sum_{j=1}^J \phi_j(\omega) \ind_{\Delta_j}(s,z),
\end{equation*}
where, for $j=1,\dots J$, we have $\Delta_j = (d_j,u_j]\times Z_j$, $0\leq d_j\leq u_j$, $Z_j\subseteq \RR$.
Then 
\begin{equation*}
\xi I(\phi)  =  \xi \sum_{j=1}^{J} \phi_{j} \mu(\Delta_j) =  \sum_{j=1}^{J} \xi \phi_{j} \mu(\Delta_j) = I(\xi \phi),
\end{equation*}
% \begin{align*}
% \xi I(\phi)  &=  \xi \sum_{j=1}^{J} \phi_{j} \mu(\Delta_j) \\
% &=   \sum_{j=1}^{J} \xi \phi_{j} \mu(\Delta_j) \\
% &= I(\xi \phi).
% \end{align*}
where $\xi \phi_j$ is $\Gg_{d_j}$-measurable since $\xi$ is $\Ff^\Lambda$-measurable. The general case follows by taking limits.
\end{proof}

\begin{remark}
% Let $\mathbb{F} = \Ff_t , t\in [0,T]$ be the filtration generated by $\mu(\Delta)$, $\Delta \in [0,t]\times \RR$. 
The random field $\mu$ is also a martingale random field with respect to $\FF$ and integration can be done as for $\GG$. However, results such as Lemma \ref{lemma:f_alpha_and integration} and the forthcoming representation would not hold. See also \cite[Remark 4.4]{Sjursen2011}. %\textbf{Some reference to why not maybe}%(In fact it can be shown that no such representation will exist for $\FF$, se
\label{remark:about_F}
\end{remark}

\section{Integral and martingale representation theorems}

In this section we prove an integral representation theorem for a random variable $\xi\in L^2(\Omega, \Ff,\Prob)$ in the setting described above. % in the It\^o type non-anticipating stochastic integral setting with respect to the filtration $\GG$. 
We freshly prove this result here for the sake of completeness. There are other similar results in the literature available. We refer for example to \cite[Theorem III.4.34]{Jacod2003}. See Remark \ref{remark:integral_representations} for further details.

Recall that $\Gg_T = \Ff_T$. Here we remark that $\Ff_T = \sigma\{ \mu(\Delta), \; \Delta \subseteq X\} =\sigma\{I(\phi),\; \phi\in\II\}$ (indeed $\mu(\Delta)=I(\ind_{\Delta})$). Denote $\mathcal{K} := \{\phi \in \mathcal{I} \big| \phi $ is $\Ff^\Lambda$-measurable, $ \phi \ind_{\Rr}$ is bounded a.e., and $ \int_0^T \int_\RR \phi_s(z)^2\ins \Lambda(ds,dz) $ is a bounded random variable$\}$. 

\begin{lemma}
For any $\phi \in \mathcal{K}$ we have
\begin{align*}
\exp\{I(\phi)\} &\in L^2(\Omega, \Ff, \Prob), \quad\quad \text{and} \\
\frac{\exp\{I(\phi)\}}{ \E\big[ \exp\{I(\phi) \} \big| \Ff^\Lambda \big]}     & \in L^2(\Omega, \Ff, \Prob).
\end{align*}
Furthermore, the random variables $\{e^{I(\phi)}, \phi \in \mathcal{K}\}$ form a total subset of $L^2(\Omega, \Ff, \Prob)$.
%\label{lemma:k_exponentials}
% \end{lemma}
\label{lemma:exponentials_dense}
\end{lemma}
\begin{proof}
The first claim is proved in \cite[Lemma 6]{Yablonski2007}, the second can be shown using arguments as in the proofs of \cite[Lemmas 4 and 6]{Yablonski2007}. The last claim is proved in \cite[Lemma 9]{Yablonski2007}.
\end{proof}
\begin{lemma}
Assume $\phi \in \mathcal{K}$. Define, for $t\in[0,T]$, 
\begin{equation*}
\zeta_t =  \exp\Big\{ \int\limits_0^t \phi_s(0) \ins dB_s +  \int\limits_0^t \int\limits_\Rr \phi_s(z) \ins\tilde{H}(ds,dz) \Big\}.
\end{equation*}
Then the following representation holds:
\begin{align}
\zeta_T =&\; \E\big[ \zeta_T \ins\big| \Ff^\Lambda\big] + \int\limits_0^T \Big[\E\big[ \frac{\zeta_T}{\zeta_s} \ins\big| \Ff^\Lambda\big]    \zeta_{s\minus}\phi_s(0) \Big] \ins dB_s         \nonumber \\
&+ \int\limits_0^T \int\limits_\Rr \Big[\E\big[ \frac{\zeta_T}{\zeta_s} \ins\big| \Ff^\Lambda\big] \zeta_{s\minus} \big(e^{\phi_s(z)} - 1 \big) \Big] \ins \tilde{H}(ds,dz).
\label{eq:zeta_representation}
\end{align}
\label{lemma:zeta_representation}
\end{lemma}
%
%\begin{remark}
\noindent
Note that the integrands in \eqref{eq:zeta_representation} are $\mathbb{G}$-predictable.
%\end{remark}
%
\begin{proof}
Let 
\begin{align}
Y_t =&\;  \frac{\zeta_t}{\E\big[ \zeta_t \ins\big| \Ff^\Lambda\big]}  \label{eq:Y_def_represention}  \\
=&\; \exp\bigg\{ \int\limits_0^t \phi_s(0) \ins dB_s + \int\limits_0^t \int\limits_\Rr \phi_s(z) \ins \tilde{H}(ds,dz) - \int\limits_0^t \frac{1}{2}\phi_s(0)^2 \ins \lambda_s^B ds\nonumber \\
&-\int\limits_0^t\int\limits_\Rr \big[ e^{\phi_s(z)}-1-\phi_s(z) \big] \ins \nu(dz) \lambda^H_s ds \bigg\}. \nonumber
\end{align}
Note that both $Y_t$ and $\zeta_t$ are elements of $L^2(\Omega, \Ff, \Prob)$ by Lemma \ref{lemma:exponentials_dense}. By It{\^o}'s formula
\begin{align}
dY_t &=Y_{t\minus} \Big( \phi_t(0) \ins dB_t  +  \int\limits_\Rr  \big( e^{\phi_t(z)}-1 \big)   \ins \tilde{H}(dt,dz)\Big), \label{eq:Y_t-ito} \\
Y_0 &= 1. \nonumber
\end{align}
Combining \eqref{eq:Y_def_represention} and \eqref{eq:Y_t-ito} the above equalities yields
\begin{align*}
\zeta_T &= \E\big[ \zeta_T \ins \big| \Ff^\Lambda\big] Y_T   \nonumber \displaybreak[0] \\ 
&= \E\big[ \zeta_T \ins\big| \Ff^\Lambda\big] \Big( 1 + \int\limits_0^T Y_{s\minus}\phi_s(0) \ins dB_s +   \int\limits_0^T\int\limits_\Rr Y_{s\minus} \big( e^{\phi_s(z)}-1 \big)   \ins \tilde{H}(ds,dz)\Big) \nonumber \displaybreak[0]  \\
&=
\begin{aligned}[t] 
&\E\big[ \zeta_T \,\big|\, \Ff^\Lambda\big] + \int\limits_0^T  \E\big[ \zeta_T \,\big|\, \Ff^\Lambda\big] Y_{s\minus} \phi_s(0) \ins dB_s \\
&+ \int\limits_0^T \int\limits_\Rr \Big[\E\big[ \zeta_T \,\big|\, \Ff^\Lambda\big] Y_{s\minus} \big(e^{\phi_s(z)} - 1 \big) \Big] \ins \tilde{H}(ds,dz)  %\label{eq:careful} 
\end{aligned}\displaybreak[0] \nonumber \\
&= 
\begin{aligned}[t]
& \E\big[ \zeta_T \,\big|\, \Ff^\Lambda\big]  + \int\limits_0^T  \E\big[ \frac{\zeta_T}{\zeta_s}  \,\big|\, \Ff^\Lambda\big] \zeta_{s\minus} \phi_s(0) \ins dB_s \\
&+ \int\limits_0^T \int\limits_\Rr \Big[\E\big[ \frac{\zeta_T}{\zeta_s} \,\big|\, \Ff^\Lambda\big] \zeta_{s\minus} \big(e^{\phi_s(z)} - 1 \big) \Big] \ins \tilde{H}(ds,dz) 
\end{aligned} \nonumber
\end{align*}
where we used Lemma \ref{lemma:f_alpha_and integration} and the equations 
\begin{equation*}
Y_s \E\big[\zeta_T \ins \big|\Ff^\Lambda \big] = Y_s\E\big[\zeta_s \,\big|\, \Ff^\Lambda \big] \E\big[\frac{\zeta_T}{\zeta_s} \ins\big| \Ff^\Lambda \big] = \zeta_s \E\big[\frac{\zeta_T}{\zeta_s} \ins\big| \Ff^\Lambda \big].  
\end{equation*}
% at \eqref{eq:careful}.
\end{proof}
% Theorem \ref{teorem:ito_representation} is classical decomposition theorem with an integral and a orthogonal component. The novelty of Theorem \ref{teorem:ito_representation} is identifying the $\Ff_T^\Lambda$-measurability of the orthogonal componenent.
\begin{teorem}
Assume $\xi \in L^2\big(\Omega,\Ff,\mathbb{P}\big)$. Then there exists a unique $\phi \in \II$ such that 
\begin{equation}
\xi = \E\big[ \xi \ins \big| \Ff^{\Lambda} \big] + \int\limits_0^T\int\limits_\RR \phi_s(z) \ins \mu(ds,dz).
\label{eq:ito_representation}
\end{equation}
\label{teorem:ito_representation}
\end{teorem}
\noindent 
Note that the two summands in \eqref{eq:ito_representation} are orthogonal. Here $\E[ \xi \ins| \Ff^{\Lambda}]$ represents the stochastic component of $\xi$ which cannot be recovered by integration on $\mu$.
\begin{proof}
At first let $\xi = \zeta(T)$, where
\begin{equation*}
\zeta(T) = \exp \Big\{\int\limits_0^T \int\limits_\RR \kappa_s(z) \ins \mu(ds,dz) \Big\}. 
\end{equation*}
From Lemma \ref{lemma:zeta_representation} the representation \eqref{eq:ito_representation} holds in this case. 

Consider a general $\xi \in L^2\big(\Omega,\mathcal{F},\mathbb{P}\big)$. Then $\xi$ can be approximated by a sequence of linear combinations of the form \eqref{eq:ito_representation} by Lemma \ref{lemma:exponentials_dense}. Let $\{\xi_n\}_{n\geq 1}$ be such a sequence. Then, by \eqref{eq:isometry_general}, we have
\begin{equation*}
\E \Big[ \big(\xi_n - \xi_m \big)^2 \Big] = \E \bigg[ \Big( \E\big[ \xi_n - \xi_m \big| \mathcal{F}^{\Lambda} \big] \Big)^2 + \int\limits_0^T\int\limits_\RR \big( \phi_s^{(n)} (z) - \phi_s^{(m)}(z) \big)^2 \ins \Lambda(ds,dz) \bigg].
\end{equation*}
Thus $\{\phi^{(n)}\}_{n\geq 1}$ is a Cauchy-sequence in $\II$, which proves existence. To prove uniqueness, suppose 
\begin{align*}
\xi &= \E\big[ \xi \ins \big| \mathcal{F}^{\Lambda} \big] + \int\limits_0^T\int\limits_\RR \phi_s(z) \ins \mu(ds,dz) \\
&= \E\big[ \xi \ins \big| \mathcal{F}^{\Lambda} \big] + \int\limits_0^T\int\limits_\RR \psi_s(z) \ins  \mu(ds,dz).
\end{align*}
Then, using \eqref{eq:isometry_general}, $\E \big[ \int_0^T\int_\RR \big( \phi_s(z) - \psi_s(z) \big)^2  \ins \Lambda(ds,dz) \big]=0$.

\end{proof}

\begin{remark}
We have here chosen to prove the above result using classical arguments well established for integrators as the Brownian motion, see e.g. \cite[Section 4]{Oksendal2005} and the Poisson random measure, see e.g. \cite{Lokka2005}. The existence of such a representation is a topic of \cite[Chapter 3]{Jacod2003}. There the result is obtained after a discussion on the solution of \emph{the martingale problem} (see \cite[Chapter 3]{Jacod2003}).

In \cite{Sjursen2011} we have instead proven this result for $\Ht$ using orthogonal polynomials and we have derived an \emph{explicit formula} for the integrand $\phi$ using the non-anticipating derivative, see \cite[Theorem 5.1]{Sjursen2011}. This result holds for more general choices of $\Lambda^H$, but with an assumption on the moments.

There are other related results in the literature. In \cite[Proposition 41]{Yablonski2007} the same representation is proved for a class of Malliavin differentiable random variables ({\`a} la Clark-Ocone type results).

% Theorem \ref{teorem:ito_representation} generalizes \cite[Proposition 41]{Yablonski2007} where the same representation was proved for a class of Malliavin differentiable random variables.
%Theorem \ref{teorem:ito_representation} was shown for $\Ht$ using orthogonal polynomials and non-anticipating derivatives in \cite{Sjursen2011}, including a explicit representation of the integrand $\phi$ \cite[Theorem 5.1]{Sjursen2011}. The result in \cite{Sjursen2011} holds for more general choices of $\Lambda^H$ but with an additional assumption on the moments. 

If $\Ff^H_T$-measurable $\xi$ are considered then representation is given
% A similar result to Theorem \ref{teorem:ito_representation} is also known for $\Ff_T^H$-measurable random variables 
in the general context of (marked) point processes, see for instance \cite[Theorem 4.12 and 8.8]{Bremaud1981} or \cite{Davis1976,Boel1975,Jacod1975}. Our result differs in the choice of filtration, which also leads to slightly different integrals. In \cite{Bremaud1981,Davis1976,Boel1975,Jacod1975} the integrator in the representation theorem are given by $H-\vartheta$ where $\vartheta$ is $\mathbb{F}^H$-predictable compensator of $H$. Our $\Lambda^H$ is not $\mathbb{F}^H$-predictable. 
\label{remark:integral_representations}
\end{remark}

\begin{teorem}
Assume $M_t$, $t\in [0,T]$, is a $\GG$-martingale. Then there exists a unique $\phi \in \II$ such that % $M$ has representation
\begin{equation*}
M_t = \E\big[ M_T \ins \big| \Ff^\Lambda \big] + \int\limits_0^t \int\limits_\RR \phi_s(z) \ins \mu(ds,dz), \quad t\in [0,T].
\end{equation*}
\label{Teorem:G_martingales}
\end{teorem}
\begin{proof}
The proof follows classical arguments as in \cite[Theorem 4.3.4]{Oksendal2005} using Theorem \ref{teorem:ito_representation}.
% \begin{align*}
% M_t &= \E\Big[ M_T \Big| \Gg_t \Big] \\
% &= \E \Big[ \E\big[ M_T \big| \Ff^\Lambda \big] + \int\limits_0^T \int\limits_\RR \phi(s,z) \ins \mu(ds,dz) \Big| \Gg_t \Big] \\
% &= \E\big[ M_T \big| \Ff^\Lambda \big] + \int\limits_0^t \int\limits_\RR \phi(s,z) \ins \mu(ds,dz).
% \end{align*}
\end{proof}

\section{BSDE: Existence and uniqueness of the solution}
\label{section:BSDEs}

Hereafter we tackle directly the question of existence and uniqueness of the solution of \eqref{eq:BSDE_intro}:
\begin{equation*}
Y_t = \xi + \int\limits_t^T g_s\big(\lambda_s,Y_s,\phi_s \big) \ins ds - \int\limits_t^T\int\limits_{\mathbb{R}} \phi_s(z)  \ins \mu(ds,dz), \quad t\in[0,T].
\end{equation*}
Indeed for the given terminal condition $\xi$ and driver (or generator) $g$, a solution is given by the couple of $\GG$-adapted processes $(Y,\phi)$ on $(\Omega, \Ff, \Prob)$ satisfying the equation above. In the sequel we characterize explicitly the functional spaces in use and the elements of the BSDE to obtain a solution. In the following section we study explicitly the case when the driver $g$ is linear.

% We want to show that given $\xi$ and $g$, there exists a process $Y$ such that $Y$ is given by a backwards stochastic differential equation and $Y(T) = \xi$. To be specific, the goal is the forthcoming equation \eqref{eq:bsde_theorem}.

Let $S$ be the space of $\GG$-adapted stochastic processes $Y(t,\omega)$, $t\in [0,T]$, $\omega \in \Omega$, such that 
\begin{equation*}
\| Y \|_S := \sqrt{ \E \big[ \sup_{0\leq t \leq T} |Y_t|^2 \big] } < \infty,
\end{equation*}
and $\Hg$ be the space of $\GG$-predictable stochastic processes $f (t,\omega)$, $t\in [0,T]$, $\omega \in \Omega$, such that
\begin{equation*}
\E \Big[ \int\limits_0^T f_s ^2 \ins ds \Big] < \infty.
\end{equation*}
Recall the definition of $\II$ in \eqref{eq:II_defined} and % and that $\lambda = (\lambda^B,\lambda^H)$.
denote $\Phi$ the space of functions $\phi :\RR \to \RR$ such that 
\begin{equation}
| \phi(0)|^2+ \int\limits_\Rr \phi(z)^2 \ins \nu(dz) < \infty.
\label{eq:Phi}
\end{equation}
\begin{definisjon}
We say that $(\xi,g)$ are \emph{standard parameters} when $\xi\in L^2\big(\Omega,\Ff,\mathbb{P}\big)$ and $g: [0,T]\times [0,\infty)^2 \times \mathbb{R}\times \Phi  \times \Omega \to \mathbb{R}$ such that $g$ satisfies (for some $K_g>0$)
\begin{align}
& g_\cdot(\lambda, Y,\phi,\cdot) \text{is $\GG$-adapted for all $\lambda\in\LL$, $Y\in S$, $\phi\in\II$,} \label{eq:f_cond0} \\
& g_\cdot(\lambda_{\cdot} ,0,0, \cdot )  \in \Hg, \text{ for all }\lambda\in\LL \label{eq:f_cond1} \displaybreak[0] \\
& \begin{aligned}[t]
\big|  &g_t \big( (\lambda^B, \lambda^H), y_1, \phi^{(1)} \big) - g_t \big((\lambda^B,\lambda^H), y_2, \phi^{(2)} \big) \big|  \leq K_g \Big( \big| y_1-y_2 \big|   \\
&+\big| \phi^{(1)}(0) - \phi^{(2)}(0) \big|\sqrt{\lambda^B}   + \sqrt{ \int\limits_\Rr  | \phi^{(1)}(z) - \phi^{(2)}(z) |  ^2 \ins  \nu(dz) } \sqrt{\lambda^H} \Big),
%& + \Big[ \int\limits_\Rr \big(\phi_t^{(1)}(z)-\phi_t^{(2)} \big)^2 \nu(dz) \lambda_t \Big]^{\frac{1}{2}}  \Big)
\end{aligned} 
\label{eq:f_cond2} \displaybreak[0] \\ 
& \text{for all }(\lambda^B,\lambda^H) \in [0,\infty)^2, y_1,y_2 \in \RR, \text{ and } \phi^{(1)},\phi^{(2)} \in \Phi \; dt\times d\Prob \text { a.e.}  \nonumber 
\end{align}
\label{definisjon:standard_parameters}
\end{definisjon}

We recall the fundamental inequality $(a_1 +a_2+\dots + a_n)^2 \leq n(a_1^2 +a_2^2 +\dots a_n^2)$, for any $n\in\mathbb{N}$ and $a_1,a_2, \dots ,a_n \in \RR$, playing an important role in the technical lemmas below.

\begin{lemma}
% Consider $\lambda=(\lambda^H,\lambda^B) \in \LL$, 
Consider $(Y,\phi), (U,\psi) \in S\times \II$.
Let $g:[0,T]\times[0,\infty)^2\times\RR\times \Phi\times\Omega \to \RR$ satisfy \eqref{eq:f_cond1} and \eqref{eq:f_cond2}. Then, for any $t\in[0,T]$, we have
\begin{align}
\E \Big[ \Big( \int\limits_t^T &  g_s(\lambda_s,Y_s,\phi_s) -  g_s(\lambda_s,U_s,\psi_s) \ins ds \Big)^2 \Big]  \leq 3 K_g^2 (T-t) \nonumber \\
&   \E \Big[ (T-t) \sup_{t\leq r \leq T} | Y_r - U_r |^2 + \int\limits_t^T\int\limits_\RR | \phi_s(z)-\psi_s(z) |^2 \Lambda(ds,dz) \Big] 
\label{eq:g_from_lips}  
\end{align}
and
\begin{align}
\E \Big[ \Big( \int\limits_t^T & \Big| g_s(\lambda_s,U_s,\psi_s) \Big| \ins ds \Big)^2 \Big] \leq  (T-t) \E \Big[ 2 \int\limits_t^T | g_s(\lambda_s,0,0) |^2 \ins ds  \nonumber \\
&  + 6 K_g^2\Big( (T-t) \sup_{t\leq r \leq T} |U_r|^2   + \int\limits_t^T\int\limits_\RR | \psi_s(z) |^2 \ins \Lambda(ds,dz)\Big) \Big].
\label{eq:g_expected_inequality}  
\end{align}
\end{lemma}

\begin{proof}
Let $t\in [0,T]$. Inequality \eqref{eq:g_from_lips} follows from the Lipschitz conditions \eqref{eq:f_cond2}:
\begin{align*}
\E \Big[ & \Big( \int\limits_t^T  g_s(\lambda_s,Y_s,\phi_s) - g_s(\lambda_s,U_s,\psi_s) \ins ds \Big)^2 \Big]  \nonumber \\ 
&\leq \begin{aligned}[t]
K_g^2 \; \E \Big[& \Big( \int\limits_t^T | Y_s-U_s |  + | \phi_s(0) - \psi_s(0) | \sqrt{\lambda_s^B} \\
&+ \sqrt{ \int\limits_\Rr \big| \phi_s(z)-\psi_s(z) \big|^2 \ins \nu(dz)} \sqrt{\lambda_s^H} \ins ds \Big)^2 \Big]
\end{aligned} \displaybreak[0]  \\
&\leq
\begin{aligned}[t]
 3 K_g^2 (T-t) \E \Big[& \int\limits_t^T \big| Y_s-U_s \big|^2  + \big| \phi_s(0) - \psi_s(0) \big|^2 \lambda_s^B \\
& + \int\limits_\Rr \big| \phi_s(z)-\psi_s(z) \big|^2 \ins \nu(dz) \lambda_s^H \ins ds  \Big] 
\end{aligned} \displaybreak[0] \\
&\leq 
\begin{aligned}[t]
3 K_g^2 (T-t) \E \Big[& (T-t) \sup_{t\leq r \leq T}  \big| Y_r-U_r \big|^2  \\
&+ \int\limits_t^T\int\limits_\RR \big| \phi_s(z)-\psi_s(z)\big|^2 \ins \Lambda(ds,dz)  \Big]. 
\end{aligned}
\end{align*}
For inequality \eqref{eq:g_expected_inequality} we have
\begin{align*}
\E \Big[ \Big( \int\limits_t^T & \Big| g_s(\lambda_s,U_s,\psi_s) \Big| \ins ds \Big)^2 \Big]  \nonumber \\
&\leq (T-t) \E \Big[  \int\limits_t^T  \big| g_s(\lambda_s,U_s,\psi_s) \big|^2  \ins ds  \Big] \nonumber \displaybreak[0] \\
& \leq (T-t) \E\Big[  \int\limits_t^T \Big( \Big| g_s(\lambda_s,0,0) \Big| + \Big| g_s(\lambda_s,U_s,\psi_s) - g_s(\lambda_s,0,0) \Big| \Big)^2 \ins ds \Big] \nonumber \displaybreak[0] \\
& \leq 2 (T-t) \E\Big[  \int\limits_t^T  \Big| g_s(\lambda_s,0,0) \Big|^2 + \Big| g_s(\lambda_s,U_s,\psi_s) - g_s(\lambda_s,0,0) \Big|^2 \ins ds \Big] \nonumber \displaybreak[0]
%& \begin{aligned}[t]
% \leq  (T-t) \E \Big[  \int\limits_t^T 2 \big|g_s(\lambda_s,0,0) \big|^2 + 6 K_g^2 \Big(  |U_s|^2 & + | \psi_s(0) |^2 \lambda^B_s \\ 
% &+ \int\limits_\Rr |\psi_s(z)|^2 \ins \nu(dz) \lambda^H_s \Big) \ins ds \Big]
% \end{aligned}  \nonumber  \\ 
% &\leq (T-t) \E \Big[ \int\limits_t^T 2 \big|g_s(\lambda_s,0,0) \big|^2 +6 K_g^2 \Big( | \psi_s(0) |^2 \lambda^B_s+  \int\limits_\Rr |\psi_s(z)|^2 \ins \nu(dz) \lambda_s^H \Big) \ins ds \Big] \nonumber \\
% &+  6 K_g^2 (T-t)^2 \E \Big[ \sup_{t\leq r\leq T} |U_s|^2 \Big]  \nonumber
\end{align*}
The result now follows from \eqref{eq:f_cond2} by proceeding as in the proof of \eqref{eq:g_from_lips} above.
%Where we used H{\"o}lder's inequality at \eqref{eq:used_holder}. %
\end{proof}

\begin{lemma}
Consider $U \in S$, $\psi,\phi \in \II$ and let $(\xi,g)$ be standard parameters. Define a stochastic process $Y_t$, $t\in [0,T]$, by
\begin{equation}
Y_t = \xi + \int\limits_t^T g_s\big(\lambda_s, U_s , \psi_s  \big) \ins ds - \int\limits_t^T\int\limits_\RR \phi_s(z) \ins \mu(ds,dz).%, \quad t\in [0,T].
\label{eq:Y_with_U}
\end{equation}
Then $Y \in S$. In particular we have
\begin{align}
\E\Big[ \sup_{t\leq r \leq T} |Y_r|^2 \Big] \leq &\; \E\Big[ 3 \xi^2 + 3\Big(\int\limits_t^T  \Big| g_s\big(\lambda_s, U_s , \psi_s \big) \Big| \ins ds\Big)^2 \nonumber \\
&+ 30  \int\limits_t^T \int\limits_\RR |\phi_s(z)|^2  \ins \Lambda(ds,dz) \big].
\label{eq:sup_Y_inequality}
\end{align}
\label{lemma:inequality_from_doob}
\end{lemma}
\begin{proof}
Directly from \eqref{eq:Y_with_U}, taking the square, we have
% From the definition of $Y$, \eqref{eq:Y_with_U}, and taking the square
\begin{equation*}
|Y_t|^2 \leq 3 \xi^2 + 3\Big(\int\limits_t^T  \Big| g_s\big(\lambda_s, U_s , \psi_s \big) \Big| ds\Big)^2 + 3\Big(\int\limits_t^T\int\limits_\RR \phi_s(z)  \ins \mu(ds,dz)\Big)^2.
\end{equation*}
In the next step we take the supremum and obtain
\begin{align*}
\E \Big[ \sup_{t\leq r \leq T} |Y_r|^2 \Big]  \leq&\; \E \Big[ 3 \xi^2 + 3 \Big( \int\limits_t^T \big|  g_s\big(\lambda_s,U_s , \psi_s \big) \big| \ins ds\Big)^2 \Big]\\
& + \E \Big[\sup_{t \leq r \leq T} 3\Big(\int\limits_r^T\int\limits_\RR \phi_s(z) \ins \mu(ds,dz)\Big)^2\Big].
\end{align*}
%
% Note that from Doob's martingale inequality, %see for instance \cite[Theorem 2.1.5]{Applebaum2009} 
We have 
\begin{align*}
\E &\Big[ \sup_{t\leq r \leq T} \Big( \int\limits_r^T\int\limits_\RR \phi_s(z) \ins \mu(ds,dz) \Big)^2 \Big] \\ 
&= \E \Big[ \sup_{t\leq r \leq T} \Big( \int\limits_t^T\int\limits_\RR \phi_s(z) \ins \mu(ds,dz)- \int\limits_t^r\int\limits_\RR \phi_s(z) \ins \mu(ds,dz) \Big)^2 \Big] \displaybreak[0] \\
&\leq \E \Big[ 2 \Big(\int\limits_t^T \int\limits_\RR \phi_s(z) \ins \mu(ds,dz) \Big)^2 + 2 \sup_{t\leq r \leq T} \Big( \int\limits_t^r \int\limits_\RR \phi_s(z) \ins \mu(ds,dz) \Big)^2 \Big] \\
&\leq 10 \E \Big[ \int\limits_t^T \int\limits_\RR \phi_s(z)^2 \ins \Lambda(ds,dz) \Big]
\end{align*}
by application of Doob's martingale inequality, see e.g. \cite[Theorem 2.1.5]{Applebaum2004}. 
Equation \eqref{eq:sup_Y_inequality} follows, and we conclude that $Y\in S$ by \eqref{eq:g_expected_inequality}.
\end{proof}
\medskip \noindent
Now let $(g,\xi)$ be standard parameters. Define the mapping% $\Theta$,
\begin{equation}
\Theta: S\times \II  \to S\times \II , \quad  \Theta(U,\psi) := (Y, \phi)
\label{eq:Theta_definition}
\end{equation}
as follows. The component $\phi$ is given by Theorem \ref{Teorem:G_martingales} as the unique element in $\II$ that provides the stochastic integral representation 
\begin{equation*}
M_t = M_0 + \int\limits_0^t\int\limits_\RR \phi_s(z) \ins \mu(ds,dz), \quad t\in [0,T], \\
\label{eq:kappa_definition}
\end{equation*}
of the martingale 
\begin{equation*}
M_t = \E\big[ \xi + \int\limits_0^T g_s\big(\lambda_s,U_s ,\psi_s \big) \ins ds \ins \big| \mathcal{G}_t \big], \quad t\in [0,T]. 
\end{equation*}
Note that $M_0 = \E\big[ \xi + \int_0^T g_s\big(\lambda_s,U_s ,\psi_s \big) \ins ds \ins \big| \Ff^\Lambda \big]$. The component $Y$ in \eqref{eq:Theta_definition} is defined by 
\begin{equation}
Y_t = \E \Big[ \xi + \int\limits_t^T g_s(\lambda_s,U_s,\psi_s)\ins  ds \ins \Big| \Gg_t \Big], \quad t\in [0,T].  \label{eq:Y_as_conditional} 
\end{equation}
Note that
\begin{align}
Y_t &= M_t - \int\limits_0^t g_s(\lambda_s, U_s,\psi_s) \ins ds \nonumber \\
&= M_0 + \int\limits_0^t\int\limits_\RR \phi_s(z) \ins \mu(ds,dz)  - \int\limits_0^t g_s(\lambda_s,U_s,\psi_s) \ins ds. \nonumber 
\end{align}
Since $Y_T=\xi$, we also have $Y_t = \xi-Y_T+Y_t$ so that
\begin{equation}
Y_t = \xi + \int\limits_t^T g_s\big(\lambda_s,U_s , \psi_s \big) \ins ds - \int\limits_t^T\int\limits_\RR \phi_s(z) \ins \mu(ds,dz).
\label{eq:Y_with_U2}
\end{equation}
Hence $Y\in S$ by Lemma \ref{lemma:inequality_from_doob} and the mapping \eqref{eq:Theta_definition} is well-defined.

We use the mapping $\Theta$ to prove that the BSDE of type \eqref{eq:BSDE_intro} admits a unique solution for the given standard parameters $(\xi,g)$.

\begin{lemma}
% Define a sequence in $S\times \II $ in a recursive manner by $Y^0 = \phi^0 = 0$ and $(Y^{n+1},\phi^{n+1}) = \Theta(Y^{n},\phi^{n},)$, for $n=0,1,\dots,$. Then
% Let $(U,\psi)\in (\SS\times \II)$ and $(Y,\phi) = \Theta(U,\psi)$. Then
%Combining \eqref{Y_recursion_bounded} and \eqref{eq:inequality_from_orthogonality} we can see that there exists a constant $K_1>0$ such that
Consider $(U^{(1)},\psi^{(1)}), (U^{(2)},\psi^{(2)}) \in S\times \II$ and define $(Y^{(1)}, \phi^{(1)}) = \Theta(U^{(1)},\psi^{(1)})$ and $(Y^{(2)}, \phi^{(2)}) = \Theta(U^{(2)},\psi^{(2)})$. Set $\bar{U}=U^{(1)}-U^{(2)}$, $\bar{\psi}=\psi^{(1)}-\psi^{(2)}$, $\bar{Y} = Y^{(1)}-Y^{(2)}$ and $\bar{\phi} = \phi^{(1)}-\phi^{(2)}$. There exists a $K>0$ such that
\begin{align}
\E &\, \bigg[ \sup_{t\leq r \leq T} \big| \bar{Y}_r \big|^2 + \int\limits_t^T \int\limits_\RR \big| \bar{\phi}_s(z) \big|^2 \ins \Lambda(ds,dz) \bigg] \nonumber \\
&\leq K (T-t) \E \bigg[ (T-t) \sup_{t\leq r \leq T} \big| \bar{U}_r \big|^2 +  \int\limits_t^T  \int\limits_\RR \big| \bar{\psi}_s(z) \big|^2 \ins \Lambda(ds,dz) \bigg], \quad t\in[0,T]. \label{eq:contraction_inequality} 
\end{align}
%
% \label{lemma:inequality_from_orthogonality}
\label{lemma:contraction_inequality} 
\end{lemma}

\begin{proof}
From \eqref{eq:Y_with_U2}, for any $t\in[0,T]$, we have
\begin{align*}
% Y^{(1)}_t - Y^{(2)}_t 
\bar{Y}_t
&= \int\limits_t^T g_s(\lambda_s,U^{(1)}_s,\psi^{(1)}_s)\ins ds- \int\limits_t^T g_s(\lambda_s,U^{(2)}_s,\psi^{(2)}_s)\ins ds \\
&-\int\limits_t^T \int\limits_{\RR} 
% \phi^{(1)}_s(z)-\phi^{(2)}_s(z)
\bar{\phi}_s(z)
\ins \mu(ds,dz).
\end{align*}
% Writing $Y_t^{n+1}$ and $Y_t^n$ as in \eqref{eq:Y_with_U2} we obtain
% \begin{align*}
% Y^{n+1}_t &- Y^{n}_t + \int\limits_t^T \int\limits_\RR \phi^{n+1}_s(z) -  \phi^{n}_s(z) \ins \mu(ds,dz) \nonumber \\
% &= \int\limits_t^T g_s(\lambda_s,Y_s^{n},\phi^n_s) - g_s(\lambda_s,Y^{n-1}_s,\phi^{n-1}_s) \ins ds.
% \label{eq:Y_n_nminusone}
% \end{align*}
%
Since 
\begin{align*}
E\Big[ & \bar{Y}_t \, \int\limits_t^T\int\limits_\RR \bar{\phi}_s(z) \ins \mu(ds,dz) \Big]  = E\Big[  \bar{Y}_t \, \E \big[ \int\limits_t^T\int\limits_\RR \bar{\phi}_s(z) \ins \mu(ds,dz) \big| \Gg_t \big] \Big] = 0,
\end{align*}
we have
\begin{align}
\E &\, \Big[ \Big( \bar{Y}_t + \int\limits_t^T\int\limits_\RR \bar{\phi}_s(z) \ins \mu(ds,dz)  \Big)^2 \Big] = \E \Big[ \big| \bar{Y}_t \big|^2 +  \int\limits_t^T\int\limits_\RR \big| \bar{\phi}_s(z) \big|^2 \ins \Lambda(ds,dz) \Big] \nonumber\displaybreak[0] \\ 
&= \E \Big[ \Big( \int\limits_t^T g_s(\lambda_s,U^{(1)}_s,\psi^{(1)}_s)\ins ds- \int\limits_t^T g_s(\lambda_s,U^{(2)}_s,\psi^{(2)}_s) \ins ds \Big)^2 \Big].
\label{eq:Y_n_nminusone}
\end{align}
%
% We apply \eqref{eq:g_from_lips} and remove $\big| Y_t^{n+1} - Y_t^{n} \big|^2$ on the left side of \eqref{eq:Y_n_nminusone} to find
% The inequality \eqref{eq:inequality_from_orthogonality} is achieved by removing $\big| Y_t^{n+1} - Y_t^{n} \big|^2$ on the left side of \eqref{eq:Y_n_nminusone} and applying \eqref{eq:g_from_lips} on the right hand side of equality \eqref{eq:Y_n_nminusone}.
We apply \eqref{eq:g_from_lips} and obtain
\begin{align}
\E &\, \Big[ \int\limits_t^T \int\limits_\RR \big| \bar{\phi}_s(z) \big|^2 \ins \Lambda(ds,dz)\Big] \leq \E \Big[ \big| \bar{Y}_t \big|^2 +  \int\limits_t^T\int\limits_\RR \big| \bar{\phi}_s(z) \big|^2 \ins \Lambda(ds,dz) \Big] \nonumber \\
&\leq 3 K_g^2 (T-t) \E \Big[ (T-t) \sup_{t\leq r\leq T} \big| \bar{U}_r \big|^2  + \int\limits_t^T\int\limits_\RR \big| \bar{\psi}_s(z) \big|^2 \ins \Lambda(ds,dz) \Big] .
\label{eq:inequality_from_orthogonality}
\end{align}
By \eqref{eq:g_from_lips}, \eqref{eq:sup_Y_inequality} and \eqref{eq:inequality_from_orthogonality} we have
\begin{align}
\E \Big[ & \sup_{t\leq r \leq T}  | \bar{Y}_r  |^2 \Big] \leq \E\Big[ 0 + 3\Big(\int\limits_t^T  \Big| g_s\big(\lambda_s, U^{(1)}_s , \psi^{(1)}_s \big)-g_s(\lambda_s,U^{(2)}_s,\psi^{(2)}_s)  \Big| ds\Big)^2 \nonumber \\
&+ 30 \Big( \int\limits_t^T \int\limits_\RR |\bar{\phi}_s(z) |^2  \ins \Lambda(ds,dz) \Big) \Big] \nonumber \displaybreak[0] \\
\leq &\; \Big( 9+90 \Big) K_g^2 (T-t)^2  \E \Big[ \sup_{t\leq r\leq T} |\bar{U}_r(z) |^2 \Big] \nonumber \\
&+ (9+90) K_g^2 (T-t) \E \Big[ \int\limits_t^T \int\limits_\RR | \bar{\psi}_s(z) |^2 \ins \Lambda(ds,dz) \Big].
\label{Y_recursion_bounded}
\end{align}

Combining \eqref{Y_recursion_bounded} and \eqref{eq:inequality_from_orthogonality} gives \eqref{eq:contraction_inequality}.
\end{proof}
The existence and uniqueness for the BSDE now follow from the above estimates:
\begin{teorem}
Let $(g,\xi)$ be standard parameters. Then  
there exists a unique couple $(Y,\phi) \in S \times \II$ such that
\begin{align}
Y_t &= \xi + \int\limits_t^T g_s\big(\lambda_s,Y_s,\phi_s \big) \ins ds - \int\limits_t^T\int\limits_\RR \phi_s(z) \ins \mu(ds,dz) \nonumber \\
&= \xi + \int\limits_t^T g_s\big(\lambda_s,Y_s,\phi_s \big) \ins ds - \int\limits_t^T \phi_s(0)\ins dB_s - \int\limits_t^T  \int\limits_{\Rr} \phi_s(z) \ins \Ht(ds,dz).
\label{eq:bsde_theorem}
\end{align}
\label{teorem:BSDE_existence_uniqueness}
\end{teorem}
\begin{proof}
Let $K$ be as in \eqref{eq:contraction_inequality}. Choose $t_1 \in [0,T)$ such that $\max\big\{ K (T-t_1)^2, K(T-t_1)\big\} < 1$. Denote $S(u,v)$ as the space consisting of the elements of $S$ equipped with the norm $\| Y \|_{S(u,v)}^2 = \E \big[ \sup_{u\leq r \leq v} | Y_r |^2 \big]$ and $\II(u,v)$ as the space of the elements of $\II$ equipped with the norm $\| \phi \|_{\II(u,v)}^2 = \E [ \int_u^v\int_\RR |\phi_s(z) |^2 \ins \Lambda(ds,dz)]$.
From \eqref{eq:contraction_inequality}, $\Theta$ is a contraction on $S(t_1,T)\times \II(t_1,T)$, and thus there exists a unique $(Y^{(1)},\phi^{(1)}) \in S(t_1,T)\times \II(t_1,T)$ such that $\Theta(Y^{(1)},\phi^{(1)}) = (Y^{(1)},\phi^{(1)})$ on $[t_1,T]$, \ie
\begin{equation*}
Y_t^{(1)} = \xi + \int\limits_t^T g_s\big(\lambda_s,Y_s^{(1)},\phi^{(1)}_s \big) \ins ds - \int\limits_t^T\int\limits_\RR \phi_s^{(1)}(z) \ins \mu(ds,dz), \quad  t\in [t_1,T] 
\end{equation*}

Take $t_2\in [0,t_1)$ so that $\max\big\{ K(t_1-t_2)^2, K(t_1-t_2)\big\} < 1$. Next, $\tilde{\phi} \in \II(t_2,t_1)$ is given by Theorem \ref{Teorem:G_martingales}, which is depending on $\tilde{U}$ and $\tilde{\psi}$, \ie
\begin{align*}
E \Big[ Y^{(1)}_{t_1} + \int\limits_0^{t_1} g_s(\lambda_s,\tilde{U}_s,\tilde{\psi}_s)\ins ds \ins \Big| \Gg_t \Big] =&\;  E \Big[ Y_t^{(1)} + \int\limits_0^{t_1} g_s(\lambda_s,\tilde{U}_s,\tilde{\psi}_s)\ins ds \ins \ins \Big| \Gg_{t_2} \Big] \\
& + \int\limits_{t_2}^t\int\limits_\RR \tilde{\phi}_s(z) \ins\mu(ds,dz) , \quad t\in[t_2,t_1],
\end{align*}
In addition, $\tilde{Y}_t$ is defined as 
\begin{align*}
\tilde{Y}_t = E \Big[ Y^{(1)}_{t_1} + \int\limits_t^{t_1} g_s(\lambda_s,\tilde{U}_s,\tilde{\psi}_s)\ins ds \ins \Big| \Gg_t \Big] , \quad t\in[t_2,t_1].
\end{align*}
Then, $\tilde{\Theta}$ can be defined by $\tilde{\Theta}(\tilde{U},\tilde{\psi})=(\tilde{Y},\tilde{\phi})$ for $(\tilde{U},\tilde{\psi})\in S(t_2,t_1)\times \II(t_2,t_1)$. 

Following the same arguments as above we conclude that $\tilde{\Theta}$ is a contraction on $S(t_2,t_1)\times \II(t_2,t_1)$ so that there exists a unique element $(Y^{(2)},\phi^{(2)})\in S(t_2,t_1)\times \II(t_2,t_1)$ such that $(Y^{(2)},\phi^{(2)})=\tilde{\Theta}(Y^{(2)},\phi^{(2)})$. Then we have
\begin{equation}
Y^{(2)}_t = Y^{(1)}_{t_1} + \int\limits_t^{t_1} g_s\big(\lambda_s,Y^{(2)}_s,\phi^{(2)}_s \big) \ins ds - \int\limits_t^{t_1}\int\limits_\RR \phi^{(2)}_s(z) \ins \mu(ds,dz), \quad t\in [t_2,t_1].
\end{equation}
Now consider 
\begin{align}
Y_t &= Y^{(1)}_t \ind_{t_1<t\leq T}(t) + Y^{(2)}_t \ind_{t_2<t\leq t_1}(t), \quad t\in [t_2,T], \nonumber \\
\phi_t &= \phi^{(1)}_t \ind_{t_1<t\leq T}(t) + \phi^{(2)}_t \ind_{t_2<t\leq t_1}(t),  \quad t\in [t_2,T]. \label{eq:iterative_scheme}
\end{align}
%
% We can now define a new BSDE on $[t_2,t_1]$, $0\leq t_2 < t_1$, with standard parameters $(Y_{t_1}^{(1)},g)$. Here $t_2$ is such that $\max\big\{ K_1 (t_1-t_2)^2, K_1(t_1-t_2)\big\} < 1$.          
% Using the same arguments as above, there exists a solution $(Y^{(2)},\phi^{(2)})$ in $S(t_2,t_1)\times \II(t_2,t_1)$ of this BSDE with parameters $(Y_{t_1}^{(1)},g)$. We combine $Y^{(1)},\phi^{(1)}$ and $Y^{(2)},\phi^{(2)}$ as the solution in $S(t_2,T)\times \II(t_2,T)$ as follows
% \begin{align}
% Y_t &= Y^{(1)}_t \ind_{t_1<t\leq T}(t) + Y^{(2)}_t \ind_{t_2<t\leq t_1}(t) \nonumber \\
% \phi_t &= \phi^{(1)}_t \ind_{t_1<t\leq T}(t) + \phi^{(2)}_t \ind_{t_2<t\leq t_1}(t). \label{eq:iterative_scheme}
% \end{align}
%
We can see that
\begin{equation}
Y_t = \xi + \int\limits_t^T g_s\big(\lambda_s,Y_s,\phi_s \big) \ins ds - \int\limits_t^T\int\limits_\RR \phi_s(z) \ins \mu(ds,dz), \quad \text{for } t\in [t_2,T].
\label{eq:XYZ}
\end{equation}
In fact, clearly \eqref{eq:XYZ} holds for $t\in [t_1,T]$. Assume $t\in(t_2,t_1]$, then 
\begin{align*}
Y_t =&\; Y_{t_1}^{(1)} +  \int\limits_t^{t_1} g_s\big(\lambda_s,Y_s^{(2)},\phi_s \big) \ins ds - \int\limits_t^{t_1}\int\limits_\RR \phi_s^{(2)}(z) \ins \mu(ds,dz) \displaybreak[0] \\
=&\; \xi + \int\limits_{t_1}^{T} g_s\big(\lambda_s,Y_s^{(1)},\phi_s \big) \ins ds - \int\limits_{t_1}^T\int\limits_\RR \phi_s^{(1)}(z) \ins \mu(ds,dz) \\
&+  \int\limits_t^{t_1} g_s\big(\lambda_s,Y^{(2)}_s,\phi_s \big) \ins ds - \int\limits_t^{t_1}\int\limits_\RR \phi_s^{(2)}(z) \ins \mu(ds,dz) \displaybreak[0] \\
=&\;  \xi + \int\limits_t^T g_s\big(\lambda_s,Y_s,\phi_s \big) \ins ds - \int\limits_t^T\int\limits_\RR \phi_s(z) \ins \mu(ds,dz).
\end{align*}
Proceed iteratively. Eventually, there is a step $n$ such that $\max\big\{ K(t_n-t_{n+1})^2, K(t_n-t_{n+1})\big\} < 1$ for $t_{n+1}=0$ (here $t_0=T)$. Then we conclude and have found a (unique) couple $(Y,\phi) \in S(0,T)\times \II(0,T)=S\times \II$ such that \eqref{eq:bsde_theorem} holds.
\end{proof}

\begin{remark}
The initial point $Y_0$ of the solution $Y$ is \emph{not} necessarily a (deterministic) constant. From the definition of $\GG$ and \eqref{eq:Y_as_conditional}, we see that $Y_0$ is a square integrable $\Ff^\Lambda$-measurable random variable. To be specific we have:
\begin{equation*}
Y_0 = \E \Big[ \xi + \int\limits_0^T g_s(\lambda_s,Y_s,\phi_s) \ins ds \ins \Big| \Ff^\Lambda \Big].
\end{equation*}
\end{remark}

\section{Linear BSDE's and a comparison theorem}

In the case of a linear driver the BSDE with Brownian motion or L{\'e}vy processes have an explicit representation. A similar represention holds in our case.  %The proof follows from the same ideas as \cite[Theorem 4.1]{Halle2010} and \cite[Theorem 6.2.2]{Pham2009}
\begin{teorem}
Assume we have the following BSDE:
\begin{align}
-dY_t =&\; \Big[ A_tY_t + C_t+ E_t(0) \phi_t(0) \sqrt{\lambda_t^B} + \int\limits_\Rr E_t(z) \phi_t (z) \ins \nu(dz) \sqrt{ \lambda_t^H} \Big] \ins dt  \nonumber \\
 &-\phi_t(0) \ins dB_t   -\int\limits_\Rr \phi_t(z) \ins \tilde{H}(dt,dz), \quad Y_T= \xi,
\label{eq:linearBSDE}
\end{align}
where the coefficients satisify
\begin{enumerate}
\renewcommand{\theenumi}{\roman{enumi})}
\renewcommand{\labelenumi}{\theenumi}
% \item $\phi\in \II$, 
\item $A$ is a bounded stochastic process, there exists $K_A > 0$ such that $|A_t | \leq K_A$ for all $t\in[0,T]$ $\Prob$-a.s.,
\label{list:boundedAB}
\item $C\in \Hg$, 
\label{list:CinH2}
\item $E \in \II$,
\item There exists $K_E > 0$ such that $ 0 \leq E_t(z) < K_E z $ for $z\in\Rr$, and $| E_t(0) | < K_E$ $\;dt\times d \Prob$-a.e.
\label{list:finite_E_integral1}
% \item There exists $\epsilon > 0$ such that $E_t(z) > -1 + \epsilon$ for all $t,z \neq 0$.
% \label{list:finite_E_integral2}
% \item "Suitable condition to ensure that the $\tilde{H}(ds,dz)$-integral in \eqref{eq:linear_ito} is a square integrable martingale"
\end{enumerate}
Then \eqref{eq:linearBSDE} has a unique solution $(Y,\phi)$ in $S\times \II$ and $Y$ has representation
\begin{equation*}
Y_t = \E \Big[ \xi \Gamma_T(t) + \int\limits_t^T \Gamma_s(t) C_s \ins ds \ins \Big| \Gg_t \Big], \quad t\in[0,T],
\end{equation*}
where
\begin{align*}
\Gamma_s(t) =&\; \exp\Big\{ \int\limits_t^s A_u -\frac{1}{2} E_u(0)^2 \ind_{\{\lambda_u^B\neq 0\}} \ins du + \int\limits_t^s E_u(0) \frac{\ind_{\{\lambda_u^B\neq 0\}} }{\sqrt{\lambda_u^B}} \ins dB_u \\
& + \int\limits_t^s \int\limits_\Rr \Big[ \ln\big(1+E_u(z)\frac{\ind_{\{\lambda_u^H\neq 0\}} }{\sqrt{\lambda_u^H}}  \big) - E_u(z)\frac{\ind_{\{\lambda_u^H \neq 0\}} }{\sqrt{\lambda_u^H}} \Big] \ins \nu(dz)\lambda^H_u \ins du \\
& + \int\limits_t^s \int\limits_\Rr \ln\big(1+E_u(z)\frac{\ind_{\{\lambda_u^H\neq 0\}} }{\sqrt{\lambda_u^H}} \big)  \ins \tilde{H}(du,dz) \Big\}.
\end{align*}
\label{teorem:theorem_linear_BSDEs}
\end{teorem}
Note that $\Gamma_s(t) = \frac{\Gamma_s(0)}{\Gamma_t(0)}$.
\begin{proof}
The proof follows classical arguments, see e.g. \cite[Theorem 6.2.2]{Pham2009}. % \cite[Theorem 4.1]{Halle2010} and . 
Condition \eqref{eq:f_cond1} is guaranteed by \ref{list:CinH2}. From H\"older's inequality
\begin{align}
\int\limits_\Rr | E_t(z) \phi_t(z)| \ins\nu(dz) \sqrt{\lambda_t^H} &\leq \sqrt{\int\limits_\Rr E_t^2(z)  \ins\nu(dz)}  \sqrt{\int\limits_\Rr  \phi_t^2(z) \ins\nu(dz) } \sqrt{ \lambda_t^H} \nonumber \\ 
& \leq K_E \sqrt{ \int\limits_\Rr z^2 \ins\nu(dz)} \sqrt{\int\limits_\Rr  \phi_t^2(z) \ins\nu(dz) }\sqrt{\lambda_t^H},
\label{eq:holder_linear}
\end{align}
so from \ref{list:boundedAB} and \ref{list:finite_E_integral1} we obtain \eqref{eq:f_cond2}.
% so the conditions \eqref{eq:f_cond1} and \eqref{eq:f_cond2} are satisified with \ref{list:boundedAB}, \ref{list:CinH2} and \ref{list:finite_E_integral1}. 
It follows from Theorem \ref{teorem:BSDE_existence_uniqueness} that \eqref{eq:linearBSDE} has a unique solution $(Y,\phi) \in S\times \II$.

\smallskip \noindent
Denote $\Gamma_t = \Gamma_t(0)$. We have $\Gamma_0=1$ and It{\^o}'s formula gives us
\begin{equation}
d\Gamma_t = \Gamma_{t\minus} \Big( A_t \ins dt + E_t(0) \frac{\ind_{\{\lambda_t^B\neq 0\}} }{\sqrt{\lambda_t^B}} \ins dB_t +  \int\limits_\Rr E_t(z)\frac{\ind_{\{\lambda_t^H\neq 0\}} }{\sqrt{\lambda_t^H}} \ins \Ht(dt,dz) \Big).
\label{eq:Gamma_SDE} 
\end{equation}
Starting from \eqref{eq:Gamma_SDE}, 
\begin{align*}
\E \Big[ |\Gamma_t|^2 \Big] &\leq 
\begin{aligned}[t]
4 \E \Big[& 1 + \Big( \int\limits_0^t \Gamma_{s\minus} A_s \ins ds \Big)^2  + \Big(\int\limits_0^t \Gamma_{s\minus} E_s(0) \frac{\ind_{\{\lambda_s^B\neq 0\}} }{\sqrt{\lambda_s^B}} \ins dB_s \Big)^2  \\
& + \Big( \int\limits_0^t\int\limits_\Rr \Gamma_{s\minus}  E_s(z)\frac{\ind_{\{\lambda_s^H\neq 0\}} }{\sqrt{\lambda_s^H}} \ins \Ht(ds,dz)\Big)^2 \Big] 
\end{aligned} \displaybreak[0] \\
& \leq 
\begin{aligned}[t]
4 \E \Big[& 1 + T \int\limits_0^t  |\Gamma_{s\minus}|^2 A^2_s \ins ds + \int\limits_0^t  | \Gamma_{s\minus}  E_s(0) |^2 \ins ds \\
&+\int\limits_0^t \int\limits_{\Rr} |\Gamma_{s\minus}|^2 K_E^2 z^2 \ins \nu(dz) \ins ds \Big] 
\end{aligned} \\
& \leq K_\Gamma \E \Big[ 1 + \int\limits_0^t |\Gamma_{s\minus}|^2 \ins ds \Big]
\end{align*}
for some $K_\Gamma>0$, since $A$ and $E(0)$ are bounded and $z^2$ is integrable with respect to $\nu$. We conclude that $\Gamma_t \in L^2(\Omega, \Ff, \Prob)$ for all $t\in [0,T]$ by Gronwall's inequality.
\medskip \noindent
By It{\^o}'s formula we have % and denoting $\Gamma_t^0 = \Gamma_t$
\begin{align}
d( Y_t\Gamma_t) =&\; \Gamma_{t\minus} \Big(\big[  -A_t Y_t - C_t - E_t(0) \phi_t(0) \sqrt{\lambda^B_t} \big] \ins dt + \phi_t(0) \ins dB_t \nonumber \\
&- \int\limits_\Rr E_t(z) \phi_t(z) \ins \nu(dz) \sqrt{\lambda_t^H} \ins dt +\int\limits_\Rr  \phi_t(z) \ins \Ht(dt,dz) \Big) \nonumber \\
&+ Y_{t\minus} \Gamma_{t\minus} \Big(  A_t \ins dt + E_t(0) \frac{\ind_{\{\lambda_t^B\neq 0\}} }{\sqrt{\lambda_t^B}} \ins dB_t + \int\limits_\Rr E_t(z)\frac{\ind_{\{\lambda_t^H\neq 0\}} }{\sqrt{\lambda_t^H}} \ins \Ht(dt,dz) \nonumber \Big) \\
&+ \Gamma_{t\minus} \Big(  E_t(0) \phi_t(0)  \frac{\ind_{\{\lambda_t^B\neq 0\}} }{\sqrt{\lambda_t^B}} \lambda_t^B \ins dt + \int\limits_\Rr E_t(z) \phi_t(z) \frac{\ind_{\{\lambda_t^H\neq 0\}} }{\sqrt{\lambda_t^H}} \ins H(dt,dz) \Big)  \nonumber \displaybreak[0] \\
=&\; 
\begin{aligned}[t]
&- \Gamma_{t\minus} C_t \ins dt + \Big[ \Gamma_{t\minus} \phi_t(0) + Y_{t\minus} E_t(0) \frac{\ind_{\{\lambda_t^B\neq 0\}} }{\sqrt{\lambda_t^B}} \Big] \ins dB_t  
+ \int\limits_\Rr \Big[ \phi_t(z) \Gamma_{t\minus} \\
&+ Y_{t\minus} \Gamma_{t\minus} E_t(z)\frac{\ind_{\{\lambda_t^H\neq 0\}} }{\sqrt{\lambda_t^H}} + \Gamma_{t\minus} \phi_t(z) E_t(z)\frac{\ind_{\{\lambda_t^H\neq 0\}} }{\sqrt{\lambda_t^H}}  \Big] \ins\tilde{H}(dt,dz). 
\end{aligned}
\label{eq:linear_ito}
\end{align}
Hence $Y_t\Gamma_t + \int_0^t \Gamma_s C_s \ins ds$, $t\in [0,T]$, is a $\GG$-martingale so that 
\begin{align*}
Y_t \Gamma_t + \int\limits_0^t \Gamma_s C_s \ins ds &= \E \Big[ Y_T \Gamma_T + \int\limits_0^T \Gamma_s C_s \ins ds \ins \Big| \Gg_t\Big] \\
Y_t \Gamma_t &= \E \Big[ \xi \Gamma_T + \int\limits_t^T \Gamma_s C_s \ins ds \ins \Big| \Gg_t\Big] \\
Y_t &= \E \Big[ \xi \Gamma^t_T + \int\limits_t^T \Gamma_s(t) C_s \ins ds \ins \Big| \Gg_t \Big].
\end{align*}
(Recall that $\Gamma_s(t)=\frac{\Gamma_s}{\Gamma_t}$).
\end{proof}

\begin{teorem} % [Comparison Theorem]
Let $(g^{(1)},\xi^{(1)})$ and $(g^{(2)},\xi^{(2)})$ be two sets of standard parameters for the BSDE's with solutions $(Y^{(1)},\phi^{(1)})$, $(Y^{(2)},\phi^{(2)})\in S\times \II$. Assume that
\begin{equation*}
g^{(2)}_t(\lambda,y,\phi,\omega) = f_t\Big(y, \phi(0) \kappa_t(0) \sqrt{\lambda^B}  ,\int\limits_\Rr \phi(z) \kappa_t(z) \ins \nu(dz) \sqrt{\lambda^H},\omega  \Big)
\end{equation*}
where $\kappa \in \II$ satisfies condition \ref{list:finite_E_integral1} from Theorem \ref{teorem:theorem_linear_BSDEs} and $f$ is a function $f: [0,T] \times \mathbb{R} \times \mathbb{R} \times \mathbb{R}\times \Omega \to \mathbb{R}$ which satisfies, for some $K_f>0$,
\begin{equation}
% \item $q\to h(t,y,\sigma,q)$ is non-decreasing for all $(t,\sigma,z) \in [0,T]\times \mathbb{R} \times \mathbb{R}$.
|f_t(y,b,h) - f_t(y',b',h')|  \leq K_f\Big( |y-y'| + |b-b'| + |h-h'| \Big), \label{eq:h_lipschitz} \\
\end{equation}
$dt\times d\Prob$ a.e. and 
\begin{equation*}
\E \Big[ \int\limits_0^T |f_t(0,0,0)|^2 \ins dt \Big] < \infty. % \label{eq:h_integrable} 
\end{equation*}
% 
% \begin{align}
% % \item $q\to h(t,y,\sigma,q)$ is non-decreasing for all $(t,\sigma,z) \in [0,T]\times \mathbb{R} \times \mathbb{R}$.
% |f_t(y,b,h) - f_t(y',b',h')| & \leq K_f\Big( |y-y'| + |b-b'| + |h-h'| \Big), \quad dt\times d\Prob \text{ a.e.} \label{eq:h_lipschitz}, \\
% \E \Big[ \int\limits_0^T |f_t(0,0,0)|^2 \ins dt \Big] &< \infty. % \label{eq:h_integrable} 
% \end{align}
%
If $\xi^{(1)} \leq \xi^{(2)}$ $\Prob$-a.s. and $g^{(1)}_s(\lambda_s,Y^{(1)}_s,\phi^{(1)}_s ) \leq g^{(2)}_s(\lambda_s,Y^{(1)}_s,\phi^{(1)}_s )$ $dt\times d\Prob$-a.e., then
\begin{equation*}
Y^{(1)}_t \leq Y^{(2)}_t  \quad dt\times d\Prob \text{-a.e.}
\end{equation*}
\end{teorem}
It can be shown that $g^{(2)}$ does indeed satisfy conditions \eqref{eq:f_cond0}-\eqref{eq:f_cond1}-\eqref{eq:f_cond2} in Definition \ref{definisjon:standard_parameters}, recall in particular \eqref{eq:holder_linear}.
\begin{proof}
Define $\bar{Y}_t := Y^{(2)}_t-Y^{(1)}_t$, $\bar{\phi}_t := \phi^{(2)}_t-\phi^{(1)}_t$, 
\begin{align*}
\phi^{(2,H)}_t(z) &:= \phi_t^{(2)} \ind_{\{z\neq 0\}} + \phi_t^{(1)}\ind_{\{z=0\}},\\
\phi^{(2,B)}_t(z) &:= \phi_t^{(2)} \ind_{\{z= 0\}} + \phi_t^{(1)}\ind_{\{z\neq 0\}},
\end{align*}
and
\begin{align*}
C_t &:= g^{(2)}_t (\lambda_t,Y^{(1)}_t,\phi_t^{(1)}) - g^{(1)}_t(\lambda_t,Y^{(1)}_t,\phi_t^{(1)}),   \displaybreak[0] \\
A_t &:= \frac{ g^{(2)}_t(\lambda_t,Y^{(2)}_t,\phi_t^{(1)}) - g^{(2)}_t(\lambda_t,Y^{(1)}_t,\phi_t^{(1)}) }{  \bar{Y}_t }\ind_{\{ \bar{Y}_t \neq 0\}}, \displaybreak[0]   \\
D_t &:= \frac{ g^{(2)}_t(\lambda_t,Y^{(2)}_t,\phi_t^{(2,H)}) - g^{(2)}_t(\lambda_t,Y^{(2)}_t,\phi_t^{(1)})}{ \int\limits_\Rr \kappa_t(z) \bar{\phi}_t(z) \ins \nu(dz) \sqrt{\lambda_t^H}  } \ind_{\{\int\limits_\Rr \kappa_t(z) \bar{\phi}_t(z)  \ins \nu(dz) \sqrt{\lambda_t^H}  \neq 0 \} }, \displaybreak[0] \\
F_t &:= \frac{ g^{(2)}_t(\lambda_t,Y^{(2)}_t,\phi^{(2,B)}_t) - g^{(2)}_t(\lambda_t,Y^{(2)}_t,\phi_t^{(1)})}{  \kappa_t(0) \bar{\phi}_t(0)  \sqrt{\lambda_t^B}  } \ind_{\{ \kappa_t(0) \bar{\phi}_t(0)  \sqrt{\lambda_t^B}  \neq 0 \} }.
\end{align*}
Then
\begin{align}
-d \bar{Y}_t =&\; \Big[ A_t \bar{Y}_t + C_t + F_t \kappa_t(0) \bar{\phi}_t(0)  \sqrt{\lambda_t^B} + D_t \int\limits_\Rr \kappa_t(z) \bar{\phi}_t(z) \ins \nu(dz) \sqrt{\lambda_t^H} \Big]dt \nonumber \\
&-\bar{\phi}_t(0) \ins dB_t - \int\limits_\Rr \bar{\phi}_t(z) \ins\Ht(dt,dz),  \nonumber \\
\bar{Y}_T =&\; \xi_2-\xi_1. \label{eq:bar_Y_comparison}
\end{align}
The processes $A$, $D$ and $F$ are bounded due to the Lipschitz condition \eqref{eq:h_lipschitz}, and $C \in \Hg$ since it is a difference of functions in $\Hg$. It follows that $F\kappa(0)+ D\kappa(z)\ind_{\Rr}(z)$ satisfies \ref{list:finite_E_integral1} in Theorem \ref{teorem:theorem_linear_BSDEs}.

Thus the assumptions of Theorem \ref{teorem:theorem_linear_BSDEs} are satisfied. The BSDE in \eqref{eq:bar_Y_comparison} has solution 
\begin{equation*}
\bar{Y}_t = \E\Big[ \bar{\xi} \Gamma^t_T + \int\limits_t^T \Gamma_s(t) C_s \ins ds \ins\Big| \Gg_t \Big] 
\end{equation*}
which is positive a.s. since $\bar{\xi}$, $\Gamma$ and $C$ are all positive a.s. 

\end{proof}

\section{Sufficient stochastic maximum principle} 

Here we show an application of the BSDE, proving sufficient conditions for an optimal control problem with both $\GG$ and $\FF$-predictable controls. This problem cannot be solved with dynamic programming methods since the state process is, in general, not Markovian.
We consider the optimization problem associated to the performance functional
\begin{equation}
J(u) = \E \Big[ \int\limits_0^T f_t(\lambda_t,u_t,X_{t\minus}) \ins dt + l(X_T) \Big],
\label{eq:performance_functional}
\end{equation}
%
% Here $f:[0,T]\times [0,\infty)^2\times \UU\times \RR \to \RR$, $\UU \subseteq \RR$ is a closed, convex set, and $l:\RR\to \RR$. Furthermore the functions $(t,\lambda,u,x) \to f_t(\lambda,u,x)$ and $x \to l(x)$ are differentiable in $x$ and $l$ is concave.
where $l(x,\omega)$, $x\in \RR$, $\omega\in\Omega$ is a stochastic function concave and differentiable in $x$ a.s. and $f_t(\lambda, u,x,\omega)$, $t\in [0,T]$, $\lambda \in [0,\infty)^2$, $u\in\UU$, $x\in \RR$, $\omega\in\Omega$ is a stochastic function differentiable in $x$ for a.s. Here $\UU \subseteq \RR$ is a closed, convex set. 
The state process $X_t$, $t\in [0,T]$, has the form
\begin{align}
dX_t &= b_t(\lambda_t,u_t, X_{t\minus}) \ins dt + \int\limits_\RR \kappa_t(z,\lambda_t, u_t,X_{t\minus}) \ins \mu(dt,dz),  \label{eq:state_process}\\
X_0 &\in \RR, \nonumber
\end{align}
where $b_t(\lambda,u,x)$ and $\kappa_t(z,\lambda,u,x)$, $t\in [0,T]$, $\lambda\in [0,\infty)^2$, $z\in\RR$, $u\in\UU$, $x\in\RR$ are $\FF$-adapted stochastic processes differentiable in $x$ a.s. We denote these derivatives $\partial_x b$ and $\partial_x \kappa$ respectively. 
The stochastic process $u_t$, $t\in[0,T]$, is the control. We have the following definition

\begin{definisjon}
The admissible controls are c\`agl\`ad stochastic processes $u: [0,T]\times \Omega \to \UU$, such that $X$ \eqref{eq:state_process} has a unique strong solution,
\begin{equation}
\E \Big[ \int\limits_0^T | f_t( \lambda_t,u_t,X_{t\minus})|^2  \ins dt + | l(X_T) | + | \partial_x l(X_T) |^2 \Big] < \infty,
\label{eq:L^2_performance_functional}
\end{equation}
and for some $K_1>0$ we have
\begin{align}
\Big| \partial_x \kappa_t(0,\lambda_t,u_t,X_{t\minus}) \Big| \sqrt{\lambda^B_t}   &\leq K_1 \quad dt\times d\Prob \text{-a.e}, \label{eq:kappa_der1}\\
\int\limits_{\Rr} \big( \partial_x \kappa_t(z,\lambda_t,u_t,X_{t\minus}) \big)^2\nu(dz) \sqrt{\lambda^H_t}   &\leq K_1 \quad dt\times d\Prob \text{-a.e}, \label{eq:kappa_der2} \\
\big| \partial_x b_t(\lambda_t,u_t,X_{t\minus}) \big| &\leq K_1 \quad dt\times d\Prob \text{-a.e} \label{eq:b_der2}.
\end{align}

The admissible controls are either $\GG$-predictable or $\FF$-predictable and we denote these sets as $\Ag$ and $\Af$ respectively. The couple $(u,X)$ is called an admissible pair.
% We denote $\Ag$ for the $\GG$-predictable admissible controls and $\Af$ for the $\FF$-predictable admissible controls.
\end{definisjon}
Naturally $\Af \subset \Ag$. Remark also that $X$ is $\GG$-adapted if $u \in \Ag$ and $X$ is $\FF$-adapted if $u\in \Af$. Given the performance functional $J$ \eqref{eq:performance_functional} we aim to find an optimal control depending on the information available:
\begin{align}
\sup_{u\in \Ag} J(u) \label{eq:GG_op_problem}\\
\sup_{u\in \Af} J(u) \label{eq:FF_op_problem}.
\end{align}

\medskip
For a detailed discussion on the existence of a strong solution to \eqref{eq:state_process} we refer to \cite{Jacod1979,Jacod1979b}. However the following conditions are sufficient \cite{Jacod1979}: for $u$ a c\`agl\`ad stochastic process
there exist a $K_2>0$ such that
\begin{align}
% labeling refers to the conditions in Jacod1979
\big| \kappa_t(0,\lambda_t,u_t,x)-\kappa_t(0,\lambda_t,u_t,x') \big| &\leq K_2 \big| x-x' \big|  \quad \Prob \text{-a.s.}, \label{eq:Au2} \\ 
% \int\limits_0^t \big| \kappa_t(0,1,u_s,\lambda_s)^2 \big| \ins \lambda_s^B ds &< \infty \quad \Prob \text{-a.s.} \\
\big| \kappa_t(z,\lambda_t,u_t,x) - \kappa_t(z,\lambda_t,u_t,x') \big| &\leq K_2 |z| | x-x' | \quad \text{for $z\neq 0$ $\Prob$-a.s.}, \label{eq:Av2} \\
\big| b_t(\lambda_t,u_t,x) - b_t(\lambda_t,u_t,x') \big| &\leq K_2 | x-x' | \quad \Prob \text{-a.s.}, \label{eq:Aw2} \displaybreak[0] \\
\int\limits_0^T\int\limits_\RR |\kappa_s(z,\lambda_s,u_s,a)|^2 \ins \Lambda(ds,dz) &< \infty \quad \Prob \text{-a.s.}, \label{eq:Au5+Av5} \\
% \big| b(x,u_t,\lambda_t) - b(x',u_t,\lambda_t) \big| & < K \big| x'-x \big| \quad \Prob \text{-a.s.} \label{eq:cond3}\\
\int\limits_0^T \big| b_s(\lambda_s,u_s,a) \big| \ins ds & \leq \infty   \quad \Prob \text{-a.s.}, \label{eq:Aw5}
\end{align}
for some $a\in \RR$, all $t\in [0,T]$ and all $x,x' \in \mathbb{R}$.

\medskip
We define the Hamiltonian, $\Ham: [0,T]\times [0,\infty)^2\times\UU\times \RR\times \RR\times \Phi\times \Omega \to \RR$ (where $\Phi$ is defined in \eqref{eq:Phi}), by 
\begin{align*}
\Ham_t(\lambda,u,x,y,\phi) =&\; f_t(\lambda,u,x) + b_t(\lambda,u,x) y + \kappa_t(0,\lambda,u,x) \phi(0) \lambda^B \\
&+ \int\limits_\Rr \kappa_t(z,\lambda,u,x) \phi(z) \ins \lambda^H \ins \nu(dz).
\end{align*}
Corresponding to the admissible pair $(u,X)$ is the couple $(Y,\phi)$, which is the solution to the BSDE of type \eqref{eq:BSDE_intro}
\begin{align}
Y_t &= \partial_x l(X_T)+ \int\limits_t^T \partial_x \Ham_s(\lambda, u_s, X_{s\minus}, Y_{s\minus},\phi_s)\ins ds - \int\limits_t^T \int\limits_\RR \phi_s(z) \ins \mu (ds,dz), \nonumber \\
Y_T &= \partial_x l(X_T).
\label{eq:adjoint_equation}
\end{align}
Here $\partial_x \Ham_t = \frac{\partial}{\partial x} \Ham_t(\lambda, u, x,y, \phi)$ and we note that $\Ham$ is differentiable in $x$ by the assumptions on $f$, $g$ and $\kappa$. The above conditions, \eqref{eq:L^2_performance_functional}-\eqref{eq:kappa_der1}-\eqref{eq:kappa_der2}-\eqref{eq:b_der2}, ensure that the pair $\big( \partial_x \Ham, \partial_x l(X_T)\big)$ are standard parameters (Definition \ref{definisjon:standard_parameters}). By Theorem \ref{teorem:BSDE_existence_uniqueness} the BSDE \eqref{eq:adjoint_equation} has a unique solution $(Y,\phi)$. % In particular \eqref{eq:kappa_der1}, \eqref{eq:kappa_der2} together with Cauchy's inequality and \eqref{eq:b_der2} ensure that the Lipschitz conditions \eqref{eq:f_cond2} are satisfied.

\smallskip
In the sequel we set $\hat{b}_s = b_s(\lambda_s,\hat{u}_s,\hat{X}_{s\minus})$, etc. for the coefficients associated with the admissible pair $(\hat{u},\hat{X})$ with solution $(\hat{Y},\hat{\phi})$ of the adjoint equation \eqref{eq:adjoint_equation}. Set $b_s = b_s(\lambda_s,u_s,X_{s\minus})$ and so forth for the coefficents associated to another arbitrary, admissible pair $(u,X)$. In addition $\hat{\Ham}_s(u,x)= \Ham_s(\lambda_s,u,x,\hat{Y}_{s\minus},\hat{\phi}_{s})$. % and $\Ham_s = \Ham_s(X_s,u_s,\hat{Y}_s,\hat{\phi}_s,\lambda_s)$. %We remark that for the  both are functions of $(\hat{Y},\hat{\phi}$
\begin{teorem}
Let $\hat{u}\in\Ag$. %Denote the corresponding state process as $\hat{X}$ \eqref{eq:state_process} and the solution of the adjoint equation \eqref{eq:adjoint_equation} as $(\hat{Y},\hat{\phi})$. 
Assume that
\begin{equation}
\E \Big[ \int\limits_0\int\limits_\RR \big|  \hat{Y}_{s\minus} \big( \hat{\kappa}_s(z)-\kappa_s(z) \big) \big|^2 +  \big| \big(\hat{X}_{s\minus} - X_{s\minus} \big) \hat{\phi}_s(z) \big|^2 \ins \Lambda(ds,dz) \Big] < \infty
\label{:eq:optimization_integrability}
\end{equation}
for all $u\in\Ag$. If 
\begin{equation}
h_t(x) = \max_{u\in \UU}  \Ham_t(\lambda_t,u, x, \hat{Y}_{t\minus}, \hat{\phi}_t ) % = \sup_{u\in \Ag} \Ham_t(\hat{X}_t, u_t, \hat{Y}_t, \hat{\phi}_t )  
\label{eq:Ham_with_u_in_U}
\end{equation} 
exists and is a concave function in $x$ for all $t\in [0,T]$ $\Prob$-a.s., and
\begin{equation}
\Ham_t(\lambda_t, \hat{u}_t,\hat{X}_{t\minus}, \hat{Y}_{t\minus}, \hat{\phi}_t ) = h_t(\hat{X}_t) % \sup_{u\in\UU} \Ham_t(\lambda_t,u,\hat{X}_{t\minus}, \hat{Y}_{t\minus}, \hat{\phi}_t ) 
\label{eq:u_GG_sup}
\end{equation}
for all $t\in [0,T]$, then $\hat{u}$ is optimal for \eqref{eq:GG_op_problem} and $(\hat{u},\hat{X})$ is an optimal pair. 
\label{teorem:optimal_control_GG}
\end{teorem}

\begin{proof}
We proceed as in \cite{Framstad2004}. Recall that for $l$ concave and differentiable we have $l(x_2) - l(x_1) \geq \partial_x l(x_2)(x_2-x_1)$, $x_1,x_2 \in \RR$. Thus, by It\^o's formula, \eqref{:eq:optimization_integrability} and the fact that $\hat{X}_0-X_0 = 0$, we have
\begin{align*}
\E &\Big[ l(\hat{X}_T) - l(X_T) \Big] \\
\geq&\; \E \Big[ \partial_x l(\hat{X}_T) \big(\hat{X}_T - X_T \big)\Big] \\
=&\; \E \Big[ \hat{Y}_T \big(\hat{X}_T - X_T \big)\Big] \displaybreak[0]\\
=&\; 
\begin{aligned}[t]
&\E \Big[ \int\limits_0^T - \big( \hat{X}_{s\minus} - X_{s\minus} \big) \partial_x \hat{\Ham}_s(\hat{u}_s,\hat{X}_{s\minus}) + \hat{Y}_{s\minus} \big(\hat{b}_s- b_s \big) \ins ds \\
&+ \int\limits_0^T\int\limits_{\RR} \Big\{ \hat{Y}_{s\minus} \big( \hat{\kappa}_s(z)-\kappa_s(z) \big) + \big(\hat{X}_{s\minus} - X_{s\minus} \big) \hat{\phi}_s(z) \Big\} \ins \mu(ds,dz) \\
&+ \int\limits_0^T \big( \hat{\kappa}_s(0)-\kappa_s(0) \big) \hat{\phi}_s(0) \ins \lambda^B_s \ins ds 
+ \int\limits_0^T\int\limits_\Rr \big( \hat{\kappa}_s(z)-\kappa_s(z) \big) \hat{\phi}_s(z) \ins H(ds,dz) \Big] 
\end{aligned}\displaybreak[0] \\
=&\;
\begin{aligned}[t]
& \E \Big[ \int\limits_0^T -\big( \hat{X}_{s\minus} - X_{s\minus} \big) \partial_x \hat{\Ham}_s(\hat{u}_s,\hat{X}_{s\minus}) + \hat{Y}_{s\minus} \big(\hat{b}_s- b_s \big) \ins ds \\
& + \int\limits_0^T\int\limits_\RR \big( \hat{\kappa}_s(z)-\kappa_s(z)\big) \hat{\phi}_s(z) \ins \Lambda(ds,dz) \Big].
\end{aligned}
\end{align*}
We remark that $\phi$ is integrable with respect to $H \times\Prob$ by \eqref{eq:Av2}, Cauchy's inequality and \eqref{:eq:optimization_integrability}.
Furthermore, from the Hamiltonian, we have
\begin{align*}
\E& \Big[ \int\limits_0^T \big\{ \hat{f}_s - f_s \big\} \ins ds \Big] \displaybreak[0] \\
&= \begin{aligned}[t]
&\E\Big[ \int\limits_0^T \Big\{ \hat{\Ham}_s(\hat{u}_s,\hat{X}_{s\minus}) -\hat{\Ham}_s(u_s,X_{s\minus}) - \big(\hat{b}_s -b_s\big) \hat{Y}_{s\minus}\\
&- \big(\hat{\kappa}_s(0)-\kappa_s(0) \big)\hat{\phi}_s(0) \lambda^B_s - \int\limits_\Rr \big( \hat{\kappa}_s(z)-\kappa_s(z)\big) \hat{\phi}_s(z) \ins \nu(dz) \lambda^H_s  \Big\} \ins ds \Big] \\
\end{aligned} \displaybreak[0] \\
&=\begin{aligned}[t] &\E\Big[ \int\limits_0^T \Big\{ \hat{\Ham}_s(\hat{u}_s,\hat{X}_{s\minus}) - \hat{\Ham}_s(u_s,X_{s\minus}) -\big(\hat{b}_s -b_s\big) \hat{Y}_{s\minus} \Big\} \ins ds \\ 
&- \int\limits_0^T\int\limits_\RR \big( \hat{\kappa}_s(z)-\kappa_s(z)\big)\hat{\phi}_s(z) \ins \Lambda(ds,dz) \Big].
\end{aligned}
\end{align*}
Hence
\begin{align}
J(\hat{u})-J(u) \geq&\; \E \Big[ \int\limits_0^T \Big\{ \hat{\Ham}_s(\hat{u}_s,\hat{X}_{s\minus}) \nonumber \\
&-\hat{\Ham}_s(u_s,X_{s\minus})-\big( \hat{X}_{s\minus} - X_{s\minus} \big) \partial_x \hat{\Ham}_s(\hat{u}_s,\hat{X}_{s\minus}) \Big\} \ins ds \Big]. \label{optimal_inequality} 
%&\geq \begin{aligned}[t]
%\E \Big[& \int\limits_0^T \Ham(X_s,\hat{u}_s,\hat{Y}_s,\hat{\phi}_s,\lambda_s)-\Ham(\hat{X}_s,\hat{u}_s,\hat{Y}_s,\hat{\phi}_s,\lambda_s) \\
%&+\Ham(\hat{X}_s,\hat{u}_s,\hat{Y}_s,\hat{\phi}_s,\lambda_s)-\Ham(X_s,u_s,\hat{Y}_s,\hat{\phi}_s,\lambda_s)   \ins ds \Big] 
%\end{aligned} \nonumber \\
%&= \E \Big[ \int\limits_0^T \Ham(X_s,\hat{u}_s,\hat{Y}_s,\hat{\phi}_s,\lambda_s)-\Ham(X_s,u_s,\hat{Y}_s,\hat{\phi}_s,\lambda_s) \ins ds \Big]
\end{align}
The integrand % $\Ham(X_s,\hat{u}_s,\hat{Y}_s,\hat{\phi}_s,\lambda_s)-\Ham(X_s,u_s,\hat{Y}_s,\hat{\phi}_s,\lambda_s)$ 
in \eqref{optimal_inequality} is non-negative $dt \times d\Prob$-a.e. by the maximality of $\hat{u}$ \eqref{eq:u_GG_sup} and the concavity of $h_t$, see \cite[page 108]{Seierstad1987}. Hence $\hat{u}$ is also an optimal control by inequality \eqref{optimal_inequality}. We sketch the last part of the argument for completeness. From \eqref{eq:Ham_with_u_in_U} and \eqref{eq:u_GG_sup} we have $h_t(\hat{X}_{t\minus})=\hat{\Ham}_t(\hat{u_t},\hat{X}_{t\minus})$. Thus
\begin{equation}
\hat{\Ham}_t(u,x)-\hat{\Ham}_t(\hat{u}_t,\hat{X}_{t\minus}) \leq h_t(x) - h_t(\hat{X}_{t\minus}), \quad \text{for all }(t,u,x).
\label{eq:Framstad21} 
\end{equation}
To prove that the integrand in \eqref{optimal_inequality} is non-negative it is sufficient to show that almost surely
\begin{equation}
 h_t(X_{t\minus}) - h_t(\hat{X_{t\minus}}) -\partial_x \hat{\Ham}_t(\hat{u}_t,\hat{X}_{t\minus}) \big( X_{t\minus} - \hat{X}_{t\minus}\big) \leq 0.
\label{eq:Framstad22}
\end{equation}
Fix $t\in [0,T]$. Since $x\to h_t(x)$ is concave, it follows by a separating hyperplane argument that there exists $a\in \RR$ such that
\begin{equation}
 h_t(x)- h_t(\hat{X}_{t\minus})-a(x-\hat{X}_{t\minus}) \leq 0, \quad\text{for all $x$}.
\label{eq:Framstad23} 
\end{equation}
Define 
\begin{equation*}
\rho(x) :=  \hat{\Ham}_t(u_t,x)- \hat{\Ham}_t(\hat{u}_t,\hat{X}_{t\minus})-a(x-\hat{X}_{t\minus}).
% \label{eq:Framstad23} 
\end{equation*}
By \eqref{eq:Framstad21} and \eqref{eq:Framstad23} $\rho(x) \leq 0$ for all $x$. Clearly $\rho(\hat{X}_{t\minus})=0$. Hence $\partial_x \rho(\hat{X}_{t\minus}) = 0$ so that $\partial_x \hat{\Ham}_t(\hat{X}_{t\minus},\hat{u}_t) =a$. Substituting into \eqref{eq:Framstad23} we obtain \eqref{eq:Framstad22}.
\end{proof}

% Note tha
% Note that, as for Theorem \ref{teorem:BSDE_existence_uniqueness}, the solutions of the BSDE \eqref{eq:adjoint_equation}

Recall that the solution of the BSDE \eqref{eq:adjoint_equation} is $\GG$-adapted. However, the other coefficients in \eqref{eq:u_GG_sup} are $\FF$-adapted whenever $u\in \Af$. We use this fact to find an optimal $\FF$-predictable control via projections.
We keep the notation used in the proof of Theorem \ref{teorem:optimal_control_GG}.

\begin{teorem}
Let $\hat{u}\in \Af$. Denote the corresponding state process as $\hat{X}$ with solution $(\hat{Y},\hat{\phi})$ of the adjoint equation \eqref{eq:adjoint_equation}. Assume \eqref{:eq:optimization_integrability} holds.
% Denote the corresponding state process as $\hat{X}$ and the solution of the adjoint equation \eqref{eq:adjoint_equation} as $(\hat{Y},\hat{\phi})$.
% Assume sufficient integrability conditions and the existence of a solution to \eqref{eq:adjoint_equation}.
%Assume $\hat{u}\in \Af$. Denote the corresponding state process $\hat{X}$ and the solution of the adjoint equation \eqref{eq:adjoint_equation} as $(\hat{Y},\hat{\phi})$. 
Denote
\begin{align*}
\Ham_t^{\FF}&(\lambda_t,u,x,\hat{Y}_{t\minus}, \hat{\phi}_t ) :=  \E\Big[ \Ham_t(\lambda_t, u, x,\hat{Y}_{t\minus}, \hat{\phi}_t )  \ins \big| \Ff_t \Big] \nonumber \\
=&\; % \sup_{u\in \UU} \Big\{ 
f_t(\lambda_t,u,x) +  b_t(\lambda_t,u,x) \E\big[ \hat{Y}_{t\minus} \ins \big| \Ff_t \big] + \kappa_t(0,\lambda_t, u,x) 
 \E\big[ \hat{\phi}_t(0) \ins \big| \Ff_t \big] \nonumber \\
&+ \int\limits_\Rr \kappa_t(z,\lambda_t,u,x) \E \big[ \phi_t(z) \ins \big| \Ff_t \big] \ins \lambda_t^H \ins \nu(dz) % \Big\}
\end{align*}
for all $t\in [0,T]$. If 
\begin{equation}
h_t^{\FF}(x) = \max_{u\in \UU}  \Ham_t^{\FF} \big(\lambda_t,u, x, \hat{Y}_{t\minus}, \hat{\phi}_{t} \big)   
\label{eq:h_t^FF}
\end{equation} 
exists and is a concave function in $x$ for all $t\in[0,T]$, and
\begin{equation}
\Ham_t^{\FF}(\lambda_t,\hat{u}_t,\hat{X}_t,\hat{Y}_{t\minus}, \hat{\phi}_t ) =h_t^{\FF}( \hat{X}_t),
\label{eq:u_FF_sup} 
\end{equation}
then $(\hat{u},\hat{X})$ is an optimal pair for \eqref{eq:FF_op_problem}.% for the $\FF$-predictable controls.
\label{teorem:optimal_control_FF}
\end{teorem}
\begin{proof}
The arguments in the proof of Theorem \ref{teorem:optimal_control_GG} leading to
\begin{equation}
J(\hat{u})-J(u) \geq \E \Big[ \int\limits_0^T  \hat{\Ham}_s(\hat{u}_s,\hat{X}_{s\minus}) -\hat{\Ham}_s(u_s,X_{s\minus})-\big( \hat{X}_{s\minus} - X_{s\minus} \big) \partial_x \hat{\Ham}_s(\hat{u}_s,\hat{X}_{s\minus})  \ins ds \Big]
\label{eq:FF_op_calculation2}
\end{equation}
still hold. %We interchange the order of integration and use conditional expectations. 
Since $\hat{u}$ and $u$ are $\FF$-predictable controls the only coefficients in the integrand in \eqref{eq:FF_op_calculation2} that are not $\FF$-adapted are the solution of the adjoint equation $(\hat{Y},\hat{\phi})$ so that
\begin{align}
\E \big[ & \int\limits_0^T  \hat{\Ham}_s(\hat{u}_s,\hat{X}_{s\minus}) -\hat{\Ham}_s(u_s,X_{s\minus}) - \partial_x \hat{\Ham}_s(\hat{u}_s,\hat{X}_{s\minus}) \big( \hat{X}_{s\minus} - X_{s\minus} \big) \ins ds \Big] \displaybreak[0] \nonumber \\
=&\;
\begin{aligned}[t]
& \E \Big[  \int\limits_0^T \hat{f}_s - f_s + \big( \hat{b}_s-b_s \big) \E\big[ \hat{Y}_{s\minus} \ins \big| \Ff_s \big] + \big(\hat{X}_{s\minus}-X_{s\minus}\big)  \partial_x f_s + \partial_x b_s \E\big[ \hat{Y}_{s\minus} \ins \big| \Ff_s \big] \ins ds \Big] \nonumber \\
&+\E \Big[ \int\limits_0^T \int\limits_\RR \Big\{ \big(\hat{\kappa}_s(z) -\kappa_s(z) \big) \E \big[ \hat{\phi}_s(z) \ins \big| \Ff_s \big] \\
&+ \big(\hat{X}_{s\minus}-X_{s\minus}\big) \partial_x \hat{\kappa}_s(z)  \E \big[ \hat{\phi}_s(z) \ins \big| \Ff_s \big] \Big\} \ins \Lambda(ds,dz) \Big] 
\end{aligned}\nonumber \displaybreak[0]\\
&= \E \Big[ \int\limits_0^T \hat{\Ham}_s^\FF(\hat{u}_s,\hat{X}_{s\minus}) - \hat{\Ham}_s^\FF(u_s,X_{s\minus}) -\partial_x \hat{\Ham}_s^\FF(\hat{u}_s,\hat{X}_{s\minus}) \big( \hat{X}_{s\minus} - X_{s\minus} \big) \ins ds \Big].  \label{eq:FF_op_calculation}
\end{align}
The integrand % $\Ham(X_s,\hat{u}_s,\hat{Y}_s,\hat{\phi}_s,\lambda_s)-\Ham(X_s,u_s,\hat{Y}_s,\hat{\phi}_s,\lambda_s)$ 
in \eqref{eq:FF_op_calculation} is non-negative $dt \times d\Prob$-a.e. by the maximality of $\hat{u}$ \eqref{eq:u_FF_sup} and the concavity of $h_t^{\FF}$ \eqref{eq:h_t^FF}. %$\tilde{\Ham}$. 
The argument is the same as in the proof of Theorem \ref{teorem:optimal_control_GG}.
% The integrands on the right hand side of \eqref{eq:FF_op_calculation} is positive $dt \times d\Prob$-a.e. by the maximality of $\hat{u}$ \eqref{eq:u_FF_sup}. Hence the inequality \eqref{eq:FF_op_calculation2} shows that $\hat{u}$ is a optimal solution.
\end{proof}

For a study on necessary maximum principles in the case of time-changed L\'evy noise we refer to \cite{Sjursen2013}, where techniques of non-anticipating stochastic derivatives are used, see \cite{DiNunno2010}.

\section{Optimal portfolio problems}

Here we show how investments in financial assets can be modeled within our framework with a state process suitable for various optimization problems. First we setup the general framework, then consider the specific problems of mean-variance hedging and utility maximization.

We consider two assets, a risk free asset $R$ and a risky asset $S$ defined by
\begin{align*}
dR_t &= \rho_t R_{t\minus} \ins dt, \quad &R_0=1, \\ 
dS_t &= \alpha_t S_{t\minus} \ins dt + S_{t\minus} \int\limits_\RR  \psi_t(z) \ins \mu(dt,dz),\quad &S_0>0.  
\end{align*}
Models of this type include \eqref{eq:Carr_model} (the model from \cite{Carr2003}) and \eqref{eq:Brownian_stochastic_volatility}. Here $\alpha$ and $\rho$ are $\FF$-adapted stochastic processes with $\alpha, \rho: [0,T]\times \Omega \to \RR$ and $\psi\in \II$ is an $\FF$-adapted random field. We assume that $\rho$ is bounded.
Let $z^{(R)}_t$ denote the units of $R$ held at time $t$ and $z^{(S)}_t$ the number of units of $S$ held at time $t$. The wealth process $X_t$, $t\in [0,T]$, is the value of the assets held, 
\begin{equation}
X_t = z^{(R)}_t R_t + z^{(S)}_t S_t.  
\label{eq:wealth_equation}
\end{equation}
We assume that the portfolio is self-financing, \ie that
\begin{equation}
dX_t = z^{(R)}_t dR_t + z^{(S)}_t dS_t.
\label{eq:self_financing}
\end{equation}
Let $u_t=z^{(S)}_t S_t$, $t\in[0,T]$, denote the amount of wealth invested in the risky asset $S$. By \eqref{eq:wealth_equation} and \eqref{eq:self_financing} the wealth equation is given by
\begin{equation}
dX_t = \big[ \rho_t X_t + (\alpha_t-\rho_t)u_t \big] \ins dt + u_{t} \int\limits_\RR \psi_t(z) \ins \mu(dt,dz).
\label{eq:wealth_differential}
\end{equation}
%
% Let $u_t$ denote the amount of wealth invested in the risky asset $S$. If the portfolio is self financing we have that the wealth equation is given by
% \begin{equation*}
% dX_t = \big[ \rho X_t + (\alpha_t-\rho_t)u_t \big] \ins dt + u_{t} \int\limits_\Rr \big[ e^z-1 \big] \ins \Ht(ds,dz).
% \end{equation*}

Clearly assumptions \eqref{eq:kappa_der1}, \eqref{eq:kappa_der2} and \eqref{eq:b_der2} are satisfied. We assume that \eqref{eq:wealth_differential} admits a strong solution. For this we see that the sufficient conditions \eqref{eq:Au2}, \eqref{eq:Av2}, \eqref{eq:Aw2}  are satisfied and together with
\begin{align*}
\int\limits_0^T | \alpha_s -\rho_s | | u_s | \ins ds &< \infty \\ 
\int\limits_0^T \int\limits_\Rr |u_s \psi_s(z) |^2 \ins \Lambda(ds,dz) &<\infty
\end{align*}
$\Prob$-a.s. also \eqref{eq:Au5+Av5} and \eqref{eq:Aw5}. The SDE \eqref{eq:wealth_differential} is of Ornstein-Uhlenbeck type and has solution 
\begin{align}\label{OU}
X_t =&\;  e^{\int_0^t \rho_r \ins dr } \Big( X_0 + \int\limits_0^t e^{- \int_0^s \rho_r \ins dr} (\alpha_s-\rho_s) u_s \ins ds \nonumber \\
& + \int\limits_0^t\int\limits_\RR e^{-\int_0^s \rho_r \ins dr} u_s \psi_s(z) \ins \mu(ds,dz) \Big).
\end{align}

We consider portfolio problems of type \eqref{eq:GG_op_problem}-\eqref{eq:FF_op_problem} associated
\begin{equation}
J(u) = \E \big[ l(X_T) \big] 
\end{equation}
with $l$ as in \eqref{eq:performance_functional}. Problems of this type include utility maximization and mean-variance portfolio selection. Hedging problems are also included, since $l$ is a function of both $\omega$ and $X$ we can consider e.g. the mean-variance hedge by $l(X_T) = -(X_T-Z)^2$ for a square integrable random variable $Z$. The Hamiltonian for this class of problems is 
\begin{align}
\Ham_t(\lambda,u,x,y,\phi) =&\; \big[\rho_t x + (\alpha_t-\rho_t) u \big] y + u \psi_t(0) \phi(0)\lambda_t^B \nonumber \\
& + \int\limits_\Rr \big[ u \psi_t(z) \phi(z) ] \ins \lambda^H_t \nu(dz)
\label{eq:Hamiltonian_financial}
\end{align}
and the associated BSDE is given by
\begin{equation}
Y_t = \partial_x l(X_T) + \int\limits_t^T Y_{s\minus} \rho_s  \ins ds - \int\limits_t^T \int\limits_\RR \phi_s(z)\ins \mu(ds,dz).
\label{eq:portfolio_BSDE}
\end{equation}
By Theorem \ref{teorem:theorem_linear_BSDEs} we also have the representation
\begin{equation}
Y_t = \E \Big[ \partial_x l(X_T) \exp\big\{ \int\limits_t^T \rho_s \ins ds \big\} \Big| \Gg_t \Big].
\label{eq:Y_conditional}
\end{equation}

\subsection{The mean-variance portfolio selection}

Here we discuss mean-variance portfolio selection starting from an initial wealth $x\in\RR$, \ie solve $\inf_{u} \E \big[ (X_T- E[X_T] )^2 \big]$ with $E[X_T] =k$  for some $k\in \RR$ and controls taking values in $\UU=\RR$. For notational convenience we consider the equivalent formulation $J(u)= \E \big[ -\frac{1}{2} \big( X_T- k \big)^2 \big]$ and want to find 
\begin{equation}
\sup_{u\in\Af} J(u) = \sup_{u\in\Af} \E \Big[ -\frac{1}{2} \big( X_T- k \big)^2 \Big].
\label{eq:mean_variance_functional_FF}
\end{equation}
To solve this problem we first consider the optimization on $u\in\Ag$ with deterministic coefficients and apply Theorem \ref{teorem:optimal_control_GG}. To avoid trivial solutions we assume $\alpha_t>\rho_t$ $dt\times d\Prob$ a.e.
%
% We discuss some important implication of Theorem \ref{teorem:G_mean_variance_optimal} before showing its proof. 
\begin{teorem} Assume that $\rho$ and $\alpha$ are deterministic. Consider the feedback control $\hat{u}^\GG \in \Ag$ given by
% Assume \eqref{:eq:optimization_integrability} holds. Then the optimal $\GG$-adapted control for \eqref{eq:mean_variance_functional} is given in feedback form by 
%
\begin{equation*}
\hat{u}_t^\GG = \frac{-\big(\alpha_t-\rho_t\big)\big(A_t \hat{X}_t + C_t\big) }{A_t \big( | \psi_t(0) |^2 \lambda_t^B + \int_\Rr| \psi_t(z) | ^2 \ins \lambda_t^H \ins \nu(dz) \big) },
\end{equation*}
where $\hat X$ refers to \eqref{OU} with $\hat{u}^\GG$ and
\begin{align}
A_t =  - \exp\Big\{- \int\limits_t^T \frac{  (\alpha_s-\rho_s)^2   } {|\psi_s(0)|^2 \lambda_s^B+ \int_\Rr | \psi_s(z) |^2 \ins\lambda_s^H\ins \nu(dz)} - 2\rho_s \ins ds \Big\} \label{eq:Aexplicit} \\
C_t = k \exp\Big\{- \int\limits_t^T \frac{  (\alpha_s-\rho_s)^2   } {|\psi_s(0)|^2 \lambda_s^B+ \int_\Rr | \psi_s(z) |^2  \ins\lambda_s^H\ins \nu(dz)} - \rho_s \ins ds \Big\}. \label{eq:Cexplicit}
\end{align}
If \eqref{:eq:optimization_integrability} holds then $\hat{u}^\GG$ is optimal for $\sup_{u\in\Ag} J(u)$.
\label{teorem:G_mean_variance_optimal}
\end{teorem}
% We discuss Theorem \ref{teorem:G_mean_variance_optimal} %and \ref{teorem:Optimal_FF_mean_variance} 
% before showing the proofs. 
 % To solve problem \eqref{eq:mean_variance_functional_FF} we use the above result combined with Theorem \ref{teorem:optimal_control_FF}.
Remark that, from \eqref{eq:Aexplicit}-\eqref{eq:Cexplicit}, the processes $A$ and $C$ depend on future values of $\lambda^B$ and $\lambda^H$, hence they are $\GG$-adapted, but in general not $\FF$-adapted. % We consider Theorem \ref{teorem:Optimal_FF_mean_variance} before showing the proof of Theorem \ref{teorem:G_mean_variance_optimal}.
% It follows that the control $\hat{u}_t^\GG$ is in general not $\FF$-adapted. However, by Theorem \ref{teorem:optimal_control_FF} the optimal $\FF$-adapted control is easily found.
%
\smallskip\noindent
Observe that equation \eqref{eq:Y_conditional} gives the useful characterization of the adjoint equation
\begin{align*}
Y_t &= \E\Big[ -(X_T-k)\exp\big\{\int\limits_t^T \rho_s \ins ds \big\} \Big| \Gg_t \Big] \\
&= \big(k-X_t\big) \E\Big[ \exp\big\{\int\limits_t^T \rho_s \ins ds \big\} \Big| \Gg_t \Big] + \E\Big[ (X_t -X_T) \exp\big\{\int\limits_t^T \rho_s \ins ds \big\} \Big| \Gg_t \Big] .
\end{align*}
The study of the above representation hints that the process $Y$ takes the form $Y_t = A_t X_t + C_t$ where $A$ and $C$ are some $\GG$-adapted processes of finite variation and $A_T =-1$, $C_T=k$. Together with some results in the L\'evy case, see e.g. \cite{Framstad2004}, we can take a guess of the structure of $Y$ for an optimal control candidate and use our verification theorem to actually determine it's optimality.
\begin{proof}[Proof of Theorem \ref{teorem:G_mean_variance_optimal}]
Denote $a$ and $c$ as of the derivatives of $A$ and $C$ with respect to $t$, so that $dA_t = a_t \ins dt$ and $dC_t = c_t \ins dt$. Set
\begin{align}
\hat{u}_t &:= - \frac{ \rho_t \big( A_t \hat{X}_t + C_t \big) +A_t \rho_t \hat{X}_t +\hat{X}_t a_t +c_t}{A_t (\alpha_t-\rho_t)},
\label{eq:mean_variance_optimal_u} \\
\hat{Y}_t &:= A_t \hat{X}_t + C_t,
\label{eq:Y_guess}
\end{align}
where $A$ and $C$ are given by \eqref{eq:Aexplicit}-\eqref{eq:Cexplicit}. 
% From \eqref{eq:paper4} substituting \eqref{eq:Y_guess} the $u$ implied by the structure of $Y$ is
We need to show 
\begin{enumerate} 
\item that \eqref{eq:Y_guess} indeed is the process $Y$ in the solution of \eqref{eq:portfolio_BSDE} corresponding to the control $\hat{u}$. %with terminal condition $Y_T = X_T -k$
\label{list:mean-variance_proof1}
\item that $\hat{u}$ \eqref{eq:mean_variance_optimal_u} satisfies \eqref{eq:u_GG_sup} in Theorem \ref{teorem:optimal_control_GG}.
\label{list:mean-variance_proof2}
\end{enumerate}
To prove \ref{list:mean-variance_proof1}, we combine \eqref{eq:wealth_differential} and \eqref{eq:Y_guess} to get
\begin{align}
d\hat{Y}_t =&\; A_t \big( \rho_t \hat{X}_t +(\alpha_t-\rho_t) \hat u_t \big) \ins dt + A_t \hat u_t \int\limits_\RR \psi_t(z) \ins \mu(dt,dz) \nonumber \\
& + \hat{X}_t a_t \ins dt + c_t \ins dt; \qquad \hat Y_T= k-\hat X_T.
\label{eq:Y_version1}
\end{align}
Inserting \eqref{eq:mean_variance_optimal_u} in \eqref{eq:Y_version1} and by the fact that $A_T = -1$ and $C_T = k$, we see that 
\begin{equation}
\hat Y_t  = k- \hat X_T + \int\limits_t^T \hat Y_{s\minus} \rho_s \ins ds - \int\limits_t^T \int\limits_\RR  \ins A_s \hat{u}_s \psi_s(z) \ins\mu(ds,dz). 
\label{eq:portfolio_BSDE2}
\end{equation}
By the uniqueness of the BSDE-solution (Theorem \ref{teorem:BSDE_existence_uniqueness}) we 
see that $(\hat Y, \hat \phi)$ with 
\begin{align}
\hat {\phi}_t(z) &= A_t \hat{u}_t \psi_t(z) , \label{eq:paper3} %\\
%\hat{Y}_{t\minus} \rho_t &= - A_t \rho_t \hat{X}_t - A_t (\alpha_t-\rho_t) \hat{u}_t - \hat{X}_t a_t -c_t.  \label{eq:paper4}  
\end{align}
%
%  with \eqref{eq:paper3}-\eqref{eq:paper4} since $A_T = -1$ and $C_T = k$. 
solves \eqref{eq:portfolio_BSDE} with $\hat u$. Hence \ref{list:mean-variance_proof1} is proved. 

\smallskip
Next we study \ref{list:mean-variance_proof2}. 
The Hamiltonian \eqref{eq:Hamiltonian_financial} is a linear function in $u \in \UU =\RR$ and, composed with $(\hat Y, \hat\phi)$, it is:
\begin{align*} 
\Ham_t(\lambda,u,x,& \hat{Y}_t,\hat{\phi}_t) = \rho_t x \hat Y_t \\
&+ u \Big[ (\alpha_t-\rho_t)  \hat{Y}_t  + \psi_t(0) \hat{\phi}_t(0)\lambda_t^B + \int\limits_\Rr \big[  \psi_t(z) \hat{\phi}_t(z) ] \ins \lambda^H_t \nu(dz) \Big].
\end{align*}
To have \eqref{eq:Ham_with_u_in_U} well defined and \eqref{eq:u_GG_sup} satisfied for the control $\hat u$ we see that
\begin{equation}
 (\alpha_t-\rho_t)\hat{Y}_t  + \psi_t(0) \hat{\phi}_t(0) \lambda^B_t+ \int\limits_\Rr \big[  \psi_t(z) \hat{\phi}_t(z)  \big] \ins \lambda^H_t \nu(dz) = 0
\label{eq:Hamiltonian_repeated}
\end{equation}
is necessary and sufficient.
Indeed we can see that equation \eqref{eq:Hamiltonian_repeated} is satisfied.
Substituting \eqref{eq:Y_guess} and \eqref{eq:paper3} into \eqref{eq:Hamiltonian_repeated}, we obtain the equation:
\begin{equation}
\hat{u}_t = \frac{-(\alpha_t-\rho_t)\big(A_t\hat{X}_t+ C_t\big)}{A_t \big( | \psi_t(0) |^2 \lambda_t^B + \int_\Rr| \psi_t(z) | ^2 \ins \lambda_t^H \ins \nu(dz) \big) }.
\label{eq:expression_u}
\end{equation}
Substituting the definition of $\hat u$ \eqref{eq:mean_variance_optimal_u}, we have
%To verify part \ref{list:mean-variance_proof2}, the characterizations given by \eqref{eq:expression_u} and \eqref{eq:mean_variance_optimal_u} must coincide, \ie we must have
%
\begin{align*}
& \big(\alpha_t-  \rho_t \big)^2 \big( A_t \hat{X}_t +C_t \big)  \\
&=  \Big( 2\rho_t A_t \hat{X}_t +\rho_t C_t + \hat{X}_t a_t +c_t\Big) \Big( | \psi_t(0)| ^2 \lambda_t^B+  \int\limits_\Rr | \psi_t(z) | ^2 \ins \lambda_t^H \ins \nu(dz) \Big),
\end{align*}
which is verified once the definitions of $A$ and $C$ \eqref{eq:Aexplicit}-\eqref{eq:Cexplicit} are inserted.
Hence, in the setting of the theorem, we conclude that $\hat u^{\GG} = \hat u$  \eqref{eq:mean_variance_optimal_u}-\eqref{eq:expression_u} is optimal.
\end{proof}
We can now solve problem \eqref{eq:mean_variance_functional_FF} using similar arguments.
% We can now prove Theorem \ref{teorem:Optimal_FF_mean_variance} using similar arguments.
% To solve problem \eqref{eq:mean_variance_functional_FF} we use the above result combined with Theorem \ref{teorem:optimal_control_FF}. 
% We discuss it before showing the proofs of Theorems \ref{teorem:G_mean_variance_optimal} and \ref{teorem:Optimal_FF_mean_variance}
%
\begin{teorem}
Consider the feedback control $\hat{u}^\FF \in \Af$ given by 
%  Then the optimal $\FF$-adapted control for \eqref{eq:mean_variance_functional} is
\begin{equation}
\hat{u}^\FF_t = - \frac{(\alpha_t-\rho_t) \Big( \E \big[ A_t \big| \Ff_t ] \hat{X}_t + \E \big[ C_t \big| \Ff_t ] \Big) }{ \E\big[ A_t \big| \Ff_t \big] \big(   | \psi_t(0) |^2 \lambda_t^B + \int_\Rr| \psi_t(z) | ^2 \ins \lambda_t^H \ins \nu(dz) \big) },
\label{eq:hatu^FF_mean_variance}
\end{equation} 
where $\hat X$ refers to \eqref{OU} with $\hat{u}^\FF$ and the processes
$A$ and $C$ are given by \eqref{eq:Aexplicit}-\eqref{eq:Cexplicit}. 
If \eqref{:eq:optimization_integrability} holds then $\hat{u}^\FF$ solves \eqref{eq:mean_variance_functional_FF}.
\label{teorem:Optimal_FF_mean_variance}
\end{teorem}
Apart from a technical point involving conditional expectations, the proof is similar to Theorem \ref{teorem:G_mean_variance_optimal}.
\begin{proof}%[Proof of Theorem \ref{teorem:Optimal_FF_mean_variance}]
Let $\hat{u}$ be given by \eqref{eq:hatu^FF_mean_variance}, \ie
\begin{equation*}
\hat{u}_t = - \frac{(\alpha_t-\rho_t) \Big( \E \big[ A_t \big| \Ff_t ] \hat{X}_t + \E \big[ C_t \big| \Ff_t ] \Big) }{ \E\big[ A_t \big| \Ff_t \big] \big(   | \psi_t(0) |^2 \lambda_t^B + \int_\Rr| \psi_t(z) | ^2 \ins \lambda_t^H \ins \nu(dz) \big) }.
\end{equation*} 
The adjoint equation is given by
\begin{align*}
\hat{Y}_t &= \E\big[ A_t \big| \Ff_t \big]  \hat{X}_t + \E\big[ C_t \big| \Ff_t \big], \\
\hat{\phi}_t(z) &= \E\big[ A_t \big| \Ff_t \big]  \hat{u}_t \psi_t(z),
\end{align*}
with $A$ and $C$ given by \eqref{eq:Aexplicit} and \eqref{eq:Cexplicit}. We can verify that $\hat{Y}$, $\hat{\phi}$ is indeed the solution of \eqref{eq:portfolio_BSDE} using the same arguments as in  the proof of Theorem \ref{teorem:G_mean_variance_optimal}. First we show that $d\E[A_t | \Ff_t ]= \E[a_t | \Ff_t ]\ins dt$ and $d\E[C_t | \Ff_t ]= \E[c_t | \Ff_t ] \ins dt$, where
\begin{align*}
\E[a_t | \Ff_t ]&= \Big( \frac{  (\alpha_t-\rho_t)^2   } {|\psi_t(0)|^2 \lambda_t^B+ \int_\Rr | \psi_t(z) |^2 \ins\lambda_t^H\ins \nu(dz)} - 2\rho_t \Big) \E\big[ A_t \big| \Ff_t \big], \\
\E[c_t | \Ff_t ]&= \Big( \frac{  (\alpha_t-\rho_t)^2   } {|\psi_t(0)|^2 \lambda_t^B+ \int_\Rr | \psi_t(z) |^2 \ins\lambda_t^H\ins \nu(dz)} - \rho_t \Big) \E\big[ C_t \big| \Ff_t \big].
\end{align*}
The argument is described for $A$, the case of $C$ is identical.
% We prove that indeed $d\E[A_t | \Ff_t ]= \E[a_t | \Ff_t ]\ins dt$, the argument with $C$ is identical. It is sufficient to show that $\frac{d}{dt} \E[A_t | \Ff_t ] = \E[a_t | \Ff_t ]$ $dt\times d\Prob$ a.e. 
Define
\begin{equation*}
 J_t := \frac{  (\alpha_t-\rho_t)^2   } {|\psi_t(0)|^2 \lambda_t^B+ \int_\Rr | \psi_t(z) |^2 \ins\lambda_t^H\ins \nu(dz)} - 2\rho_t. 
\end{equation*}
Remark that
\begin{equation*}
A_t = - \exp \big\{ -\int\limits_t^T J_s \ins ds \big\}. 
\end{equation*}
Thus, we have
\begin{align*}
\E \big[ A_{t+\Delta t} \big| \Ff_{t+\Delta t} \big] &= \E \Big[ - \exp \big\{-\int\limits_{t+\Delta t}^T J_s \ins ds \big\} \Big| \Ff_{t+\Delta t} \Big] \displaybreak[0] \\
&= \exp\big\{ \int\limits_{t}^{t+\Delta t} J_s \ins ds \big\}   \E \Big[- \exp \big\{ -\int\limits_{t}^T J_s \ins ds \big\} \Big| \Ff_{t+\Delta t} \Big].
\end{align*}
Hence, since $ \E \big[ \exp \big\{ -\int_{t}^T J_s \ins ds \big\} \big| \Ff_{t+\Delta t} \big] \to \E \big[ \exp \big\{ -\int_{t}^T J_s \ins ds \big\} \big| \Ff_t \big]$ $dt\times d\Prob$ a.e as $\Delta t \to 0$, we obtain
\begin{align*}
\lim_{\Delta t \to 0^+}  \frac{ \E\big[ A_{t+\Delta t} \big| \Ff_{t+\Delta t} \big] - \E\big[ A_t \big| \Ff_t \big]}{\Delta t} &= \lim_{\Delta t \to 0^+}  \frac{ \exp\big\{ \int_t^{t+\Delta t}  J_s \ins ds\big\} -1  }{\Delta t} \E[ A_t | \Ff_t ] \\
&= J_t \E[ A_t | \Ff_t ] = \E[ a_t | \Ff_t ].
\end{align*}
Following Theorem \ref{teorem:optimal_control_FF} we define
\begin{align*}
\Ham^\FF_t (\lambda,u,x,\hat{Y}_t,\hat{\phi}_t) =&\; \rho_t x \E \big[ \hat{Y}_t \big| \Ff_t \big] + u \Big\{ \big(\alpha_t-\rho_t) \E \big[ \hat{Y}_t \big| \Ff_t \big] \\
& + \psi_t(0)\E\big[ \hat{\phi}_t(0) \big| \Ff_t \big]  \lambda^B 
+ \int\limits_\Rr \psi_t(z) \E \big[  \hat{\phi}_t(z) \big| \Ff_t \big] \ins \lambda^H \nu(dz)  \Big\}.
\end{align*}
For $\hat{u}$ to be optimal, from \eqref{eq:h_t^FF}-\eqref{eq:u_FF_sup}, it is sufficient to show that 
\begin{equation}
 \big(\alpha_t-\rho_t) \E \big[ \hat{Y}_t \big| \Ff_t \big] + \psi_t(0)\E\big[ \hat{\phi}_t(0) \big| \Ff_t \big]  \lambda^B_t 
+ \int\limits_\Rr \psi_t(z) \E \big[  \hat{\phi}_t(z) \big| \Ff_t \big] \ins \lambda^H_t \nu(dz) = 0.
\label{eq:hamFF_to_optimize}
\end{equation}
Replacing $\hat{\phi}_t(z) = A_t \hat{u}_t \psi_t(z)$ and inserting \eqref{eq:hatu^FF_mean_variance} in \eqref{eq:hamFF_to_optimize}, we obtain the desired result.
\end{proof}

\subsection{Utility maximization of final wealth}
For the utility maximization problem of the final wealth we set 
\begin{equation*}
J(u) = E\big[ U(X_T) \big] 
\end{equation*}
where $U:\RR\to\RR$ is a differentiable utility function, increasing and concave. The BSDE is given by
\begin{equation*}
Y_t = U'(X_T) + \int\limits_t^T Y_{s\minus} \rho_s  \ins ds - \int\limits_t^T \int\limits_\RR \phi_s(z)\ins \mu(ds,dz) 
\end{equation*}
(where $U'(x) = \frac{d}{dx} U(x)$). % The function $x\to h_t(x)$ is concave for all $t\in[0,T]$. 
By arguments as in the mean-variance portfolio problem, the sufficient condition from the maximum principle for optimal $\hat{u}\in \Ag$ can be reduced to 
\begin{equation*}
(\alpha_t-\rho_t) \hat{Y}_t =  \psi_t(0) \hat{\phi}_t(0) \lambda^B_t + \int\limits_\Rr \big[ \psi_t(z) \hat{\phi}_t(z) \big] \ins \lambda^H_t \nu(dz),
\end{equation*}
or
\begin{align*}
(\alpha_t-\rho_t) & \Big( U'(\hat{X}_T)+\int\limits_t^T \hat{Y}_{s\minus} \rho_s \ins ds -\int\limits_t^T \int\limits_\Rr \hat{\phi}_s(z) \ins \mu(ds,dz) \Big) \\
&=  \psi_t(0) \hat{\phi}_t(0) \lambda^B_t + \int\limits_\Rr \big[ \psi_t(z) \hat{\phi}_t(z) \big] \ins \lambda^H_t \nu(dz),
\end{align*}
where $\hat{Y}$, $\hat{\phi}$ depends on $\hat{u}$.

\section{Acknowlegdements}

We thank the expert referee for his careful reading and detailed comments. We also thank Tusheng Zhang and Etienne Pardoux for their comments at the start of this work. The research leading to these results has received funding from the
European Research Council under the European Community's Seventh Framework
Programme (FP7/2007-2013) / ERC grant agreement no [228087].

\bibliographystyle{alpha}
\phantomsection
% \addcontentsline{toc}{section}{References}
\bibliography{referanser}

\end{document}